\documentclass[letterpaper,12pt]{amsart}

\usepackage[divide={1in,*,1in}]{geometry}
\usepackage{amsfonts,amssymb}
\usepackage{mathtools}
\usepackage{stmaryrd}
\usepackage{pstricks,pst-node}
\usepackage{mywrap}
\usepackage{centerpict}
\usepackage{diagps,diagps-arrows}
\usepackage{cobordisms}

\usepackage[
		pdfauthor={Wojciech Lubawski and Krzysztof K. Putyra},
		pdftitle={Mirror links have dual odd and generalized Khovanov homology},
		colorlinks=true]{hyperref}
\usepackage[savepos]{zref}
\usepackage[all]{hypcap}		

\COBembeddedtrue
\def\COBxsize{0.4}
\def\COBysize{1.2}

\def\fntCircle{\begingroup
	\psset{unit=2ex,linewidth=0.5pt,dotsep=1pt}%
	\begin{pspicture}[shift=-0.15](-0.5,-0.5)(0.5,0.5)%
			\pscircle(0,0){0.5}
	\end{pspicture}\endgroup}%

\def\drawSaddle{%
	\begin{centerpict}(-1pt,-2pt)(21pt,18pt)%
		\psset{linewidth=0.5pt}%
		\psline(17pt,16pt)(17pt,5pt)
		\psbezier( 0pt, 5pt)( 9pt, 3pt)( 9pt, 3pt)(17pt, 5pt)
		\psline[border=1pt]( 3pt,11pt)( 3pt,0pt)
		\psbezier[border=1pt](20pt,11pt)(12pt,12pt)(11pt,15pt)(17pt,16pt)
		\psbezier( 3pt, 0pt)(11pt, 2pt)(11pt, 2pt)(20pt, 0pt)
		\psbezier( 3pt,11pt)( 9pt,12pt)( 8pt,15pt)( 0pt,16pt)
		\psline(17pt,16pt)(17pt,14pt)
		\psline( 0pt,16pt)( 0pt,5pt)
		\psline(20pt,11pt)(20pt,0pt)
		\psbezier(6.67pt,14pt)(7pt,5pt)(13pt,5pt)(13.33pt,14pt)
	\end{centerpict}%
}

\def\trefoil#1#2#3{%
	\begin{centerpict}(-1,-0.8)(1,1.1)%
		\psbezier(0,1)( 0.5,1)( 0.59641,0.23302)( 0.34641,-0.2)			
		\psbezier(-0.86603,-0.5)(-1.11503,-0.06698)(-0.5,0.4)(0,0.4)		
		\psbezier( 0.86603,-0.5)( 0.61603,-0.93302)(-0.09641,-0.63302)(-0.34641,-0.2)		
		\psbezier[border=4pt](0,1)(-0.5,1)(-0.59641,0.23302)(-0.34641,-0.2)			
		\psbezier[border=4pt]( 0.86603,-0.5)( 1.11503,-0.06698)( 0.5,0.4)(0,0.4)		
		\psbezier[border=4pt](-0.86603,-0.5)(-0.61603,-0.93302)( 0.09641,-0.63302)( 0.34641,-0.2)		
		\psline[linewidth=0.5pt,arrowsize=5pt,arrowinset=0.4]{->}(0.625,0.01519)(0.225,0.70801)
		\psline[linewidth=0.5pt,arrowsize=5pt,arrowinset=0.4]{->}(0.3,-0.55)(-0.5,-0.55)
		\psline[linewidth=0.5pt,arrowsize=5pt,arrowinset=0.4]{->}(-0.625,0.01519)(-0.225,0.70801)
		\uput[ 30]( 0.475, 0.275){$1$}
		\uput[-90]( 0.000,-0.550){$2$}
		\uput[150](-0.475, 0.275){$3$}
	\end{centerpict}%
}

\def\trefoilshadow{%
	\psbezier(0,1)( 0.5,1)( 0.59641,0.23302)( 0.34641,-0.2)
	\psbezier(-0.86603,-0.5)(-1.11503,-0.06698)(-0.5,0.4)(0,0.4)
	\psbezier( 0.86603,-0.5)( 0.61603,-0.93302)(-0.09641,-0.63302)(-0.34641,-0.2)
	\psbezier(0,1)(-0.5,1)(-0.59641,0.23302)(-0.34641,-0.2)
	\psbezier( 0.86603,-0.5)( 1.11503,-0.06698)( 0.5,0.4)(0,0.4)
	\psbezier(-0.86603,-0.5)(-0.61603,-0.93302)( 0.09641,-0.63302)( 0.34641,-0.2)
}

\def\trefoilresolution#1#2#3{%
	\trefoilshadow
	\begingroup
		\psset{linestyle=none,fillstyle=solid,fillcolor=white}%
		\pscircle( 0    ,-0.55){0.2}
		\pscircle( 0.475,0.275){0.2}
		\pscircle(-0.475,0.275){0.2}
	\endgroup
	\ifx1#1\relax
		\psbezier(0.47105,0.46555)(0.48679,0.33682)(0.53343,0.254  )(0.63713,0.17613)
		\psbezier(0.45367,0.0879 )(0.48132,0.21841)(0.42964,0.30766)(0.30286,0.34897)
	\else
		\psbezier(0.47105,0.46555)(0.48679,0.33682)(0.42964,0.30766)(0.30286,0.34897)
		\psbezier(0.45367,0.0879 )(0.48132,0.21841)(0.53343,0.254  )(0.63713,0.17613)
		\ifx0#1\relax\else
			\psline[linewidth=0.5pt,arrowsize=3pt,arrowinset=0.4]{->}(0.625,0.01519)(0.225,0.70801)
		\fi
	\fi
	\ifx1#2\relax
		\psbezier(-0.16765,-0.64072)(-0.04831,-0.58998)( 0.04688,-0.58917)( 0.1662 ,-0.6401 )
		\psbezier( 0.15071,-0.43686)( 0.05151,-0.52606)(-0.05151,-0.52606)(-0.15071,-0.43686)
	\else
		\psbezier(-0.16765,-0.64072)(-0.04831,-0.58998)(-0.05151,-0.52606)(-0.15071,-0.43686)
		\psbezier( 0.15071,-0.43686)( 0.05151,-0.52606)( 0.04688,-0.58917)( 0.1662 ,-0.6401 )
		\ifx0#2\relax\else
			\psline[linewidth=0.5pt,arrowsize=3pt,arrowinset=0.4]{->}(0.4,-0.55)(-0.5,-0.55)
		\fi
	\fi
	\ifx1#3\relax
		\psbezier(-0.47105,0.46555)(-0.48679,0.33682)(-0.53343,0.254  )(-0.63713,0.17613)
		\psbezier(-0.45367,0.0879 )(-0.48132,0.21841)(-0.42964,0.30766)(-0.30286,0.34897)
	\else
		\psbezier(-0.47105,0.46555)(-0.48679,0.33682)(-0.42964,0.30766)(-0.30286,0.34897)
		\psbezier(-0.45367,0.0879 )(-0.48132,0.21841)(-0.53343,0.254  )(-0.63713,0.17613)
		\ifx0#3\relax\else
			\psline[linewidth=0.5pt,arrowsize=3pt,arrowinset=0.4]{->}(-0.625,0.01519)(-0.225,0.70801)
		\fi
	\fi
}

\def\pictRelS{\begin{centerpict}(0,0)(1,1)
	\psset{linewidth=0.5pt,dash=1pt 1.5pt,dimen=middle}
	\pscircle(0.5,0.5){0.4}
	\psellipticarc[linewidth=0.5\pslinewidth](0.5,0.5)(0.4,0.15){180}{360}
	\psellipticarc[linewidth=0.5\pslinewidth,linestyle=dashed](0.5,0.5)(0.4,0.15){0}{180}
\end{centerpict}}

\def\pictRelT{\begin{centerpict}(-0.5,-0.7)(0.5,0.7)
	\psset{linewidth=0.5pt,dash=1pt 1.5pt,dimen=middle}
	\psellipse(0,0)(0.45,0.6)
	\psellipticarc(0.3,0)(0.4,0.45){130}{230}
	\psclip{\psellipse[linestyle=none](0.3,0)(0.4,0.45)}
		\psellipticarc(-0.3,0)(0.4,0.45){-50}{50}
	\endpsclip
	\psellipticarc[linewidth=0.5\pslinewidth](0.275,0)(0.175,0.075){180}{360}
	\psellipticarc[linewidth=0.5\pslinewidth](-0.275,0)(0.175,0.075){180}{360}
	\psellipticarc[linewidth=0.5\pslinewidth,linestyle=dashed](0.275,0)(0.175,0.075){0}{180}
	\psellipticarc[linewidth=0.5\pslinewidth,linestyle=dashed](-0.275,0)(0.175,0.075){0}{180}
\end{centerpict}}

\def\pictRelTuCircles{%
	\psset{linewidth=0.5pt,dash=1pt 1.5pt,dimen=middle}
	\psellipticarc(0.2,0.1)(0.2,0.07\psxunit){-180}{0}
	\psellipticarc(1.0,0.1)(0.2,0.07\psxunit){-180}{0}
	\psellipticarc[linewidth=0.5\pslinewidth,linestyle=dashed](0.2,0.1)(0.2,0.07\psxunit){0}{180}
	\psellipticarc[linewidth=0.5\pslinewidth,linestyle=dashed](1.0,0.1)(0.2,0.07\psxunit){0}{180}
	\psellipse(0.2,1.4)(0.2,0.07\psxunit)
	\psellipse(1.0,1.4)(0.2,0.07\psxunit)
}

\def\pictRelTuL{\begin{centerpict}(-0.1,0)(1.3,1.5)
	\pictRelTuCircles
	\psline(0.0,0.1)(0.0,1.4)\psline(0.4,0.1)(0.4,1.4)
	\psbezier(0.8,0.1)(0.8,0.7)(1.2,0.7)(1.2,0.1)
	\psbezier(0.8,1.4)(0.8,0.8)(1.2,0.8)(1.2,1.4)
\end{centerpict}}

\def\pictRelTuR{\begin{centerpict}(-0.1,0)(1.3,1.5)
	\pictRelTuCircles
	\psbezier(0.0,0.1)(0.0,0.7)(0.4,0.7)(0.4,0.1)
	\psbezier(0.0,1.4)(0.0,0.8)(0.4,0.8)(0.4,1.4)
	\psline(0.8,0.1)(0.8,1.4)\psline(1.2,0.1)(1.2,1.4)
\end{centerpict}}

\def\pictRelTuB{\begin{centerpict}(-0.1,0)(1.3,1.5)
	\pictRelTuCircles
	\psbezier(0.0,0.1)(0.0,1.0)(1.2,1.0)(1.2,0.1)
	\psbezier(0.4,0.1)(0.35,0.5)(0.85,0.5)(0.8,0.1)
	\psbezier(0.0,1.4)(0.0,1.0)(0.4,1.0)(0.4,1.4)
	\psbezier(0.8,1.4)(0.8,0.8)(1.2,0.8)(1.2,1.4)
\end{centerpict}}

\def\pictRelTuT{\begin{centerpict}(-0.1,0)(1.3,1.5)
	\pictRelTuCircles
	\psbezier(0.0,1.4)(0.0,0.5)(1.2,0.5)(1.2,1.4)
	\psbezier(0.4,1.4)(0.35,1.0)(0.85,1.0)(0.8,1.4)
	\psbezier(0.0,0.1)(0.0,0.5)(0.4,0.5)(0.4,0.1)
	\psbezier(0.8,0.1)(0.8,0.7)(1.2,0.7)(1.2,0.1)
\end{centerpict}}

\input{images/pics-tangles.tex}
\def\defPicSmallCob#1#2{%
	\expandafter\def\csname fntCob#1\endcsname{%
		\begin{centerpict}(-0.1,0)(1.1,1.2)
			\psset{linewidth=0.15pt,dotsep=1pt,dash=1pt 1.5pt,linestyle=dotted,dimen=mid}
			\psset{linewidth=0.3pt,linestyle=solid}#2
		\end{centerpict}}%
}

\def\defPicCob#1#2{\expandafter\def\csname cobInsert@#1\endcsname{#2}}%

\def\psdash#1{#1[linestyle=dashed,linewidth=0.15pt]}

\def\insertCob#1{\begin{centerpict}(-0.1,0)(1.3,2)
	\psset{xunit=1.2cm,linewidth=0.15pt,dotsep=1pt,dash=1pt 1pt,linestyle=dotted,dimen=mid}
	\psset{linewidth=0.4pt,linestyle=solid}
	\csname cobInsert@#1\endcsname
\end{centerpict}}%

\def\fntCob#1{\csname fntCob#1\endcsname}

\def\cobRIhtop{%
	\psellipticarc(0.65,0)(0.15,0.1){-90}{90}
	\psbezier(0.13523, 0.13679)(0.45,-0.10)(0.3 , 0.1)(0.65, 0.1)
	\psbezier(0.27639,-0.17889)(0.35, 0.05)(0.45,-0.1)(0.65,-0.1)
}
\def\cobRIvtop{%
	\psellipse(0.65,0)(0.15,0.1)
	\psbezier(0.27639,-0.17889)(0.4,0)(0.4,0)(0.13523,0.13679)
}
\def\cobRIhbottom{%
	\psellipticarc(0.65,0)(0.15,0.1){-90}{0}
	\psbezier(0.27639,-0.17889)(0.35,0.05)(0.45,-0.1)(0.65,-0.1)
	\psclip{\psframe[linestyle=none](0,-0.2)(0.277,1)}
		\psbezier(0.13523,0.13679)(0.45,-0.1)(0.3,0.1)(0.65,0.1)
	\endpsclip
	\psdash\psellipticarc(0.65,0)(0.15,0.1){0}{90}
	\psclip{\psframe[linestyle=none](1,-0.2)(0.276,1)}
		\psdash\psbezier(0.13523,0.13679)(0.45,-0.1)(0.3,0.1)(0.65,0.1)
	\endpsclip
}
\def\cobRIvbottom{%
	\psellipticarc(0.65,0)(0.15,0.1){-180}{0}
	\psdash\psellipticarc(0.65,0)(0.15,0.1){0}{180}
	\psclip{\pspolygon[linestyle=none](0,-0.2)(0.4,-0.2)(0.4,-0.028)(0.277,-0.028)(0.277,0.2)(0,0.2)}
		\psbezier(0.27639,-0.17889)(0.4,0)(0.4,0)(0.13523,0.13679)
	\endpsclip
	\psclip{\psframe[linestyle=none](0.2763,-0.0284)(0.3564,0.2)}
		\psdash\psbezier(0.27639,-0.17889)(0.4,0)(0.4,0)(0.13523,0.13679)
	\endpsclip
}

\defPicSmallCob{RId}{%
	\psline(0.27639,0.02111)(0.27639,0.82111)
	\psline(0.13523,0.33679)(0.13523,1.13679)
	\psline(0.80000,0.20000)(0.80000,1.00000)
	\rput(0,0.2)\cobRIvbottom
	\rput(0,1.0)\cobRIhtop
	\psbezier(0.35637,0.17164)(0.37,0.7)(0.49,0.7)(0.5,0.2)
}

\defPicSmallCob{RIh}{%
	\psline(0.27639,0.02111)(0.27639,0.82111)
	\psline(0.13523,0.33679)(0.13523,1.13679)
	\psbezier(0.35637,0.97164)(0.37,0.4)(0.8,0.4)(0.8,0.2)
	\rput(0,0.2)\cobRIhbottom
	\rput(0,1.0)\cobRIvtop
	\psbezier(0.5,1)(0.5,0.6)(0.8,0.6)(0.8,1)
}

\defPicSmallCob{RIg}{%
	\psline(0.27639,0.02111)(0.27639,0.82111)
	\psline(0.13523,0.33679)(0.13523,1.13679)
	\psbezier(0.35637,0.17164)(0.37,0.8)(0.8,0.8)(0.8,1)
	\rput(0,0.2)\cobRIvbottom
	\rput(0,1.0)\cobRIhtop
	\psbezier(0.5,0.2)(0.5,0.6)(0.8,0.6)(0.8,0.2)
}

\defPicSmallCob{RIfa}{%
	\psline(0.27639,0.02111)(0.27639,0.82111)
	\psline(0.13523,0.33679)(0.13523,1.13679)
	\psbezier(0.35637,0.97164)(0.37,0.1)(0.8,0.5)(0.8,0.2)
	\rput(0,0.2)\cobRIhbottom
	\rput(0,1.0)\cobRIvtop
	\psellipticarc(0.65,0.75)(0.2,0.2){-180}{0}
	\psbezier(0.5,1)(0.5,0.9)(0.45,0.85)(0.45,0.75)
	\psbezier(0.8,1)(0.8,0.9)(0.85,0.85)(0.85,0.75)
	\psellipticarc(0.65,0.78)(0.1,0.1){-170}{-10}
	\psellipticarc(0.65,0.65)(0.1,0.1){40}{140}
}

\defPicSmallCob{RIfb}{%
	\psline(0.27639,0.02111)(0.27639,0.82111)
	\psline(0.13523,0.33679)(0.13523,1.13679)
	\psline(0.80000,0.20000)(0.80000,1.00000)
	\rput(0,0.2)\cobRIhbottom
	\rput(0,1.0)\cobRIvtop
	\psbezier(0.35637,0.97164)(0.37,0.5)(0.49,0.5)(0.5,1)
}

\def\cobRIIhT{%
	\psbezier(0.20,-0.160)(0.52,-0.03)(0.52,-0.03)(0.95,-0.087)
	\psbezier(0.05, 0.087)(0.48, 0.03)(0.48, 0.03)(0.80, 0.160)
}

\def\cobRIIvvT{%
	\psbezier(0.2,-0.16)(0.3,-0.05)(0.25,0)(0.05, 0.087)
	\psbezier(0.8, 0.16)(0.7, 0.05)(0.75,0)(0.95,-0.087)
	\psbezier(0.45,0.07)(0.61,0.10)(0.71,-0.04)(0.55,-0.07)
	\psbezier(0.45,0.07)(0.29,0.04)(0.39,-0.10)(0.55,-0.07)
}

\def\cobRIIvhT{%
	\psbezier(0.2,-0.16)(0.3,-0.05)(0.25,0)(0.05, 0.087)
	\psbezier(0.45,0.07)(0.29,0.04)(0.39,-0.10)(0.55,-0.07)
	\psbezier(0.80, 0.160)(0.75,0.00)(0.61, 0.10)(0.45, 0.07)
	\psbezier(0.95,-0.087)(0.73,0.06)(0.71,-0.04)(0.55,-0.07)
}

\def\cobRIIhvT{%
	\psbezier(0.8, 0.16)(0.7, 0.05)(0.75,0)(0.95,-0.087)
	\psbezier(0.45,0.07)(0.61,0.10)(0.71,-0.04)(0.55,-0.07)
	\psbezier(0.20,-0.160)(0.25, 0.00)(0.39,-0.10)(0.55,-0.07)
	\psbezier(0.05, 0.087)(0.27,-0.06)(0.29, 0.04)(0.45, 0.07)
}

\def\cobRIIhhT{%
	\psbezier(0.80, 0.160)(0.75,0.00)(0.61, 0.10)(0.45, 0.07)
	\psbezier(0.95,-0.087)(0.73,0.06)(0.71,-0.04)(0.55,-0.07)
	\psbezier(0.20,-0.160)(0.25, 0.00)(0.39,-0.10)(0.55,-0.07)
	\psbezier(0.05, 0.087)(0.27,-0.06)(0.29, 0.04)(0.45, 0.07)
}

\def\cobRIIhB{%
	\psbezier(0.20,-0.160)(0.52,-0.03)(0.52,-0.03)(0.95,-0.087)
	\psclip{\psframe[linestyle=none](0,-0.2)(0.2,0.2)}%
		\psbezier(0.05, 0.087)(0.48, 0.03)(0.48, 0.03)(0.80, 0.160)
	\endpsclip
	\psclip{\psframe[linestyle=none](0.2,-0.2)(1,0.2)}%
		\psdash\psbezier(0.05, 0.087)(0.48, 0.03)(0.48, 0.03)(0.80, 0.160)
	\endpsclip
}

\def\cobRIIvvB{%
	\psclip{\pspolygon[linestyle=none](0,-0.2)(0,0.2)(0.2,0.2)(0.2,-0.0664)(0.25,-0.0664)(0.368,0)(0.631,0)(0.75,0.0664)(1,0.0664)(1,-0.2)}%
		\psbezier(0.2,-0.16)(0.3,-0.05)(0.25,0)(0.05, 0.087)
		\psbezier(0.8, 0.16)(0.7, 0.05)(0.75,0)(0.95,-0.087)
		\psbezier(0.45,0.07)(0.61,0.10)(0.71,-0.04)(0.55,-0.07)
		\psbezier(0.45,0.07)(0.29,0.04)(0.39,-0.10)(0.55,-0.07)
	\endpsclip
	\psclip{\pspolygon[linestyle=none](0.2,0.2)(0.2,-0.0664)(0.25,-0.0664)(0.368,0)(0.631,0)(0.75,0.0664)(1,0.0664)(1,0.2)}
		\psdash\psbezier(0.2,-0.16)(0.3,-0.05)(0.25,0)(0.05, 0.087)
		\psdash\psbezier(0.8, 0.16)(0.7, 0.05)(0.75,0)(0.95,-0.087)
		\psdash\psbezier(0.45,0.07)(0.61,0.10)(0.71,-0.04)(0.55,-0.07)
		\psdash\psbezier(0.45,0.07)(0.29,0.04)(0.39,-0.10)(0.55,-0.07)
	\endpsclip
}

\def\cobRIIvhB{%
	\psclip{\pspolygon[linestyle=none](0,0.2)(0.2,0.2)(0.2,-0.0664)(0.25,-0.0664)(0.368,0)(1,0.11)(1,-0.2)(0,-0.2)}
		\psbezier(0.2,-0.16)(0.3,-0.05)(0.25,0)(0.05, 0.087)
		\psbezier(0.45,0.07)(0.29,0.04)(0.39,-0.10)(0.55,-0.07)
		\psbezier(0.95,-0.087)(0.73,0.06)(0.71,-0.04)(0.55,-0.07)
	\endpsclip
	\psclip{\pspolygon[linestyle=none](0.2,0.2)(0.2,-0.0664)(0.25,-0.0664)(0.368,0)(1,0.11)(1,0.2)}
		\psdash\psbezier(0.2,-0.16)(0.3,-0.05)(0.25,0)(0.05, 0.087)
		\psdash\psbezier(0.45,0.07)(0.29,0.04)(0.39,-0.10)(0.55,-0.07)
		\psdash\psbezier(0.80, 0.160)(0.75,0.00)(0.61, 0.10)(0.45, 0.07)
	\endpsclip
}

\def\cobRIIhvB{%
	\psclip{\pspolygon[linestyle=none](1,-0.2)(1,0.0664)(0.75,0.0664)(0.631,0)(0.2,-0.02)(0.2,0.2)(0,0.2)(0,-0.2)}
		\psbezier(0.8, 0.16)(0.7, 0.05)(0.75,0)(0.95,-0.087)
		\psbezier(0.45,0.07)(0.61,0.10)(0.71,-0.04)(0.55,-0.07)
		\psbezier(0.20,-0.160)(0.25, 0.00)(0.39,-0.10)(0.55,-0.07)
		\psbezier(0.05, 0.087)(0.27,-0.06)(0.29, 0.04)(0.45, 0.07)
	\endpsclip
	\psclip{\pspolygon[linestyle=none](1,0.2)(1,0.0664)(0.75,0.0664)(0.631,0)(0.2,-0.11)(0.2,0.2)}
		\psdash\psbezier(0.8, 0.16)(0.7, 0.05)(0.75,0)(0.95,-0.087)
		\psdash\psbezier(0.45,0.07)(0.61,0.10)(0.71,-0.04)(0.55,-0.07)
		\psdash\psbezier(0.05, 0.087)(0.27,-0.06)(0.29, 0.04)(0.45, 0.07)
	\endpsclip
}

\def\cobRIIhhB{%
	\psbezier(0.95,-0.087)(0.73,0.06)(0.71,-0.04)(0.55,-0.07)
	\psbezier(0.20,-0.160)(0.25, 0.00)(0.39,-0.10)(0.55,-0.07)
	\psclip{\psframe[linestyle=none](0,0.2)(0.2,-0.2)}
		\psbezier(0.05, 0.087)(0.27,-0.06)(0.29, 0.04)(0.45, 0.07)
	\endpsclip
	\psdash\psbezier(0.80, 0.160)(0.75,0.00)(0.61, 0.10)(0.45, 0.07)
	\psclip{\psframe[linestyle=none](0.2,0.2)(1,-0.2)}
		\psdash\psbezier(0.05, 0.087)(0.27,-0.06)(0.29, 0.04)(0.45, 0.07)
	\endpsclip
}

\defPicSmallCob{R2-d0*}{%
	\psline(0.05, 0.287)(0.05, 1.087)
	\psline(0.20, 0.040)(0.20, 0.840)
	\psdash\psline(0.80, 0.360)(0.80, 0.997)
	\psline(0.80, 0.997)(0.80, 1.160)
	\psline(0.95, 0.113)(0.95, 0.913)
	\psline(0.369,0.2)(0.369,1)
	\psline(0.25,0.1336)(0.25,0.9336)
	\rput(0,0.2)\cobRIIvhB
	\rput(0,1.0)\cobRIIvvT
	\psbezier(0.631,1)(0.631,0.3)(0.75,0.3)(0.75,1.0664)
}

\defPicSmallCob{R2-d*0}{%
	\psline(0.05, 0.287)(0.05, 1.087)
	\psline(0.20, 0.040)(0.20, 0.840)
	\psdash\psline(0.80, 0.360)(0.80, 0.987)
	\psline(0.80, 0.987)(0.80, 1.160)
	\psline(0.95, 0.113)(0.95, 0.913)
	\rput(0,0.2)\cobRIIvhB
	\rput(0,1.0)\cobRIIhhT
	\psbezier(0.369,0.2)(0.369,0.9)(0.25,0.9)(0.25,0.1336)
}

\defPicSmallCob{R2-d1*}{%
	\psline(0.05, 0.287)(0.05, 1.087)
	\psline(0.20, 0.040)(0.20, 0.840)
	\psdash\psline(0.80, 0.360)(0.80, 0.997)
	\psline(0.80, 0.997)(0.80, 1.160)
	\psline(0.95, 0.113)(0.95, 0.913)
	\rput(0,0.2)\cobRIIhhB
	\rput(0,1.0)\cobRIIhvT
	\psbezier(0.631,1)(0.631,0.3)(0.75,0.3)(0.75,1.0664)
}

\defPicSmallCob{R2-d*1}{%
	\psline(0.05, 0.287)(0.05, 1.087)
	\psline(0.20, 0.040)(0.20, 0.840)
	\psdash\psline(0.80, 0.360)(0.80, 0.997)
	\psline(0.80, 0.997)(0.80, 1.160)
	\psline(0.95, 0.113)(0.95, 0.913)
	\psline(0.631,0.2)(0.631,1)
	\psline(0.75,0.2664)(0.75,1.0664)
	\rput(0,0.2)\cobRIIvvB
	\rput(0,1.0)\cobRIIhvT
	\psbezier(0.369,0.2)(0.369,0.9)(0.25,0.9)(0.25,0.1336)
}

\defPicSmallCob{R2-h0*}{%
	\psline(0.05, 0.287)(0.05, 1.087)
	\psline(0.20, 0.040)(0.20, 0.840)
	\psdash\psline(0.80, 0.360)(0.80, 0.987)
	\psline(0.80, 0.987)(0.80, 1.160)
	\psline(0.95, 0.113)(0.95, 0.913)
	\psline(0.25,0.1336)(0.25,0.9336)
	\rput(0,0.2)\cobRIIvvB
	\rput(0,1.0)\cobRIIvhT
	\psbezier(0.75,0.2664)(0.75,0.9)(0.369,0.8)(0.369,1)
	\psbezier(0.369,0.2)(0.369,0.6)(0.631,0.6)(0.631,0.2)
}

\defPicSmallCob{R2-h*1}{%
	\psline(0.05, 0.287)(0.05, 1.087)
	\psline(0.20, 0.040)(0.20, 0.840)
	\psdash\psline(0.80, 0.360)(0.80, 0.997)
	\psline(0.80, 0.997)(0.80, 1.160)
	\psline(0.95, 0.113)(0.95, 0.913)
	\psline(0.75,0.2664)(0.75,1.0664)
	\rput(0,0.2)\cobRIIhvB
	\rput(0,1.0)\cobRIIvvT
	\psbezier(0.25,0.9336)(0.25,0.3)(0.631,0.4)(0.631,0.2)
	\psbezier(0.369,1)(0.369,0.6)(0.631,0.6)(0.631,1)
}

\defPicSmallCob{R2-F}{%
	\psline(0.05, 0.287)(0.05, 1.087)
	\psline(0.20, 0.040)(0.20, 0.840)
	\psdash\psline(0.80, 0.360)(0.80, 0.997)
	\psline(0.80, 0.997)(0.80, 1.160)
	\psline(0.95, 0.113)(0.95, 0.913)
	\rput(0,0.2)\cobRIIhB
	\rput(0,1.0)\cobRIIvvT
	\psbezier(0.369,1)(0.369,0.6)(0.631,0.6)(0.631,1)
	\psbezier(0.25,0.9336)(0.25,0.2336)(0.75,0.1664)(0.75,1.0664)
	\psbezier(0.534,0.148)(0.534,0.3)(0.52,0.4)(0.4923,0.4)
	\psdash\psbezier(0.466,0.252)(0.466,0.35)(0.475,0.4)(0.4923,0.4)
}

\defPicSmallCob{R2-G}{%
	\psline(0.05, 0.287)(0.05, 1.087)
	\psline(0.20, 0.040)(0.20, 0.840)
	\psdash\psline(0.80, 0.360)(0.80, 0.933)
	\psline(0.80, 0.933)(0.80, 1.160)
	\psline(0.95, 0.113)(0.95, 0.913)
	\rput(0,0.2)\cobRIIvvB
	\rput(0,1.0)\cobRIIhT
	\psbezier(0.369,0.2)(0.369,0.6)(0.631,0.6)(0.631,0.2)
	\psbezier(0.25,0.1336)(0.25,1.0336)(0.75,0.9664)(0.75,0.2664)
	\psbezier(0.534,0.948)(0.534,0.9)(0.52,0.8)(0.4923,0.8)
	\psclip{\psframe[linestyle=none](0.4,1.1)(0.5,0.935)}
		\psbezier(0.466,1.052)(0.466,0.85)(0.475,0.8)(0.4923,0.8)
	\endpsclip
	\psclip{\psframe[linestyle=none](0.4,0.7)(0.5,0.935)}
		\psdash\psbezier(0.466,1.052)(0.466,0.85)(0.475,0.8)(0.4923,0.8)
	\endpsclip
}

\def\pictRelS{\begin{centerpict}(-0.6,-0.5)(0.6,0.5)
	\psset{linewidth=0.5pt,dash=1pt 1.5pt,dimen=middle}%
	\pscircle(0,0){0.5}
	\psellipticarc(0,0)(0.5,0.15){-180}{0}
	\psdash\psellipticarc(0,0)(0.5,0.15){0}{180}
\end{centerpict}}

\def\pictRelT{\begin{centerpict}(-0.1,-0.7)(1.3,0.7)
	\psset{linewidth=0.5pt,dash=1pt 1.5pt,dimen=middle}
	\psdash\psellipticarc(0.2,0)(0.2,0.1){0}{180}
	\psdash\psellipticarc(1,0)(0.2,0.1){0}{180}
	\psellipticarc(0.2,0)(0.2,0.1){-180}{0}
	\psellipticarc(1,0)(0.2,0.1){-180}{0}
	\psbezier(0.0,0)(0.00,0.9)(1.20,0.9)(1.2,0)
	\psbezier(0.4,0)(0.35,0.4)(0.85,0.4)(0.8,0)
	\psline[linewidth=0.3pt,arrowsize=3pt]{<-}(0.45,0.1)(0.75,0.1)
	\psbezier(0.0,0)(0.00,-0.9)(1.20,-0.9)(1.2,0)
	\psbezier(0.4,0)(0.35,-0.4)(0.85,-0.4)(0.8,0)
	\psline[linewidth=0.3pt,border=1pt,arrowsize=3pt]{->}(0.40,-0.4)(0.70,-0.1)
\end{centerpict}}

\def\pictRelTuCircles{%
	\psset{linewidth=0.5pt,dash=1pt 1.5pt,dimen=middle}
	\psellipticarc(0.2,0.1)(0.2,0.07\psxunit){-180}{0}
	\psellipticarc(1.0,0.1)(0.2,0.07\psxunit){-180}{0}
	\psdash\psellipticarc(0.2,0.1)(0.2,0.07\psxunit){0}{180}
	\psdash\psellipticarc(1.0,0.1)(0.2,0.07\psxunit){0}{180}
	\psellipse(0.2,1.4)(0.2,0.07\psxunit)
	\psellipse(1.0,1.4)(0.2,0.07\psxunit)
}

\def\pictRelTuL{\begin{centerpict}(-0.1,0)(1.3,1.5)
	\pictRelTuCircles
	\psline(0.0,0.1)(0.0,1.4)\psline(0.4,0.1)(0.4,1.4)
	\psbezier(0.8,0.1)(0.8,0.7)(1.2,0.7)(1.2,0.1)
	\psbezier(0.8,1.4)(0.8,0.8)(1.2,0.8)(1.2,1.4)
\end{centerpict}}

\def\pictRelTuR{\begin{centerpict}(-0.1,0)(1.3,1.5)
	\pictRelTuCircles
	\psbezier(0.0,0.1)(0.0,0.7)(0.4,0.7)(0.4,0.1)
	\psbezier(0.0,1.4)(0.0,0.8)(0.4,0.8)(0.4,1.4)
	\psline(0.8,0.1)(0.8,1.4)\psline(1.2,0.1)(1.2,1.4)
\end{centerpict}}

\def\pictRelTuB{\begin{centerpict}(-0.1,0)(1.3,1.5)
	\pictRelTuCircles
	\psbezier(0.0,0.1)(0.0,1.0)(1.2,1.0)(1.2,0.1)
	\psbezier(0.4,0.1)(0.35,0.5)(0.85,0.5)(0.8,0.1)
	\psline[linewidth=0.3pt,arrowsize=3pt]{<-}(0.45,0.2)(0.75,0.2)
	\psbezier(0.0,1.4)(0.0,1.0)(0.4,1.0)(0.4,1.4)
	\psbezier(0.8,1.4)(0.8,0.8)(1.2,0.8)(1.2,1.4)
\end{centerpict}}

\def\pictRelTuT{\begin{centerpict}(-0.1,0)(1.3,1.5)
	\pictRelTuCircles
	\psbezier(0.0,1.4)(0.0,0.5)(1.2,0.5)(1.2,1.4)
	\psbezier(0.4,1.4)(0.35,1.0)(0.85,1.0)(0.8,1.4)
	\psline[linewidth=0.3pt,border=1pt,arrowsize=3pt]{->}(0.40,1.0)(0.70,1.3)
	\psbezier(0.0,0.1)(0.0,0.5)(0.4,0.5)(0.4,0.1)
	\psbezier(0.8,0.1)(0.8,0.7)(1.2,0.7)(1.2,0.1)
\end{centerpict}}

\newtheorem{theorem}{Theorem}[section]
\newtheorem{lemma}[theorem]{Lemma}

\newtheorem{proposition}[theorem]{Proposition}
\newtheorem{corollary}[theorem]{Corollary}

\theoremstyle{definition}

\newtheorem{definition}[theorem]{Definition}
\newtheorem{example}[theorem]{Example}
\newtheorem{remark}[theorem]{Remark}
\newtheorem{observation}[theorem]{Observation}

\definecolor{internalLink}{rgb}{0.5,0,0}
\definecolor{citeLink}{rgb}{0,0.5,0}
\definecolor{urlLink}{rgb}{0,0,0.5}

\hypersetup{linkcolor=internalLink}
\hypersetup{citecolor=citeLink}
\hypersetup{urlcolor=urlLink}

\DeclareMathOperator{\id}{id}		

\def\blank{\raisebox{0.3ex}{\underline{\ \ }}}		
\def\cocolon{%
	\nobreak\mskip6mu plus1mu\mathpunct{}%
	\nonscript\mkern-\thinmuskip{:}%
	\mskip2mu\relax
}

\newcommand*{\bS}{\mathbb{S}^1}			
\newcommand*{\Z}{\mathbb{Z}}				
\newcommand*{\R}{\mathbb{R}}				

\newcommand*{\Aut}{\mathrm{Aut}}

\let\scalars\Bbbk
\def\invScalars{\scalars^*}

\def\Zev{\Z_{ev}}						
\def\Zodd{\Z_{odd}}					
\def\Zpi{\mathbb{Z}_{\pi}}	\def\ZpiLong{\mathbb{Z}[\pi]/(\pi^2-1)}

\makeatletter
\def\scalarsLong{\@ifstar
	{\mathbb{Z}[\permMM,\permSS,\permMS^{\pm1}]/(\permMM^2{=}\permSS^2{=}1)}%
	{\mathbb{Z}[\permMM,\permSS,\permMS^{\pm1}]/(\permMM^2=\permSS^2=1)}%
}
\makeatother

\newcommand*{\cat}[1]{\textrm{\normalfont\textbf{#1}}}		
\newcommand*{\catAdd}[1]{\mathrm{Mat}(#1)}

\newcommand*{\Mor}{\mathrm{Mor}}

\newcommand*{\F}{\mathcal{F}}
\newcommand*{\FA}{\mathcal{F}}
\newcommand*{\Fpi}{\mathcal{F}_{\!\pi}}
\newcommand*{\Fev}{\mathcal{F}_{\!ev}}

\newcommand*\Mod[1]{\cat{Mod}_{#1}}
\newcommand*{\Kom}{\textbf{Kom}}

\newcommand*{\cone}{C}
\newcommand*\Hom{\mathrm{Hom}}

\def\newcobcategory#1#2{%
	\expandafter\def\csname #1\endcsname{#2}
	\expandafter\def\csname #1L\endcsname{#2_{/\ell}}
	\expandafter\def\csname #1D\endcsname{#2_{\bullet}}
}

\newcobcategory{Cob}{\cat{Cob}}
\newcobcategory{ChCob}{\cat{ChCob}}
\newcobcategory{kChCob}{\scalars\cat{ChCob}}

\newcommand*{\vv}[2]{v_{#1}\otimes v_{#2}}
\newcommand*\vvv[1]{\vvvv#1,\endvvv,}

\def\vvvv#1,{v_{#1}\vvvvv}
\def\vvvvv#1,{%
	\ifx\endvvv#1\relax\else
		\otimes\ifx*#1\else v_{#1}\fi
	\expandafter\vvvvv\fi}

\DeclareMathOperator{\ldsum}{%
	\begin{pspicture}(-0.2em,0)(0.8em,0.8em)
		\psset{unit=1em,linewidth=0.4pt,arrowsize=4pt,arrowlength=0.8,arrowinset=0.5}%
		\psline(0,0.7)(0,0)(0.6,0)(0.6,0.7)
		\psline{->}(0.0,0.5)(0.0,0.2)
		\psline{->}(0.6,0.2)(0.6,0.5)
	\end{pspicture}%
}
\DeclareMathOperator{\rdsum}{%
	\begin{pspicture}(-0.2em,0)(0.8em,0.8em)
		\psset{unit=1em,linewidth=0.4pt,arrowsize=4pt,arrowlength=0.8,arrowinset=0.5}%
		\psline(0,0.7)(0,0)(0.6,0)(0.6,0.7)
		\psline{->}(0.6,0.7)(0.6,0.2)
		\psline{->}(0,0)(0,0.5)
	\end{pspicture}%
}

\makeatletter
\def\KhBracket{\@ifstar\KhBracketScaled\KhBracketSimple}
\def\KhCube{\@ifstar\KhCubeScaled\KhCubeSimple}
\def\KhCubeSigned{\@ifstar\KhCubeSignedScaled\KhCubeSignedSimple}
\makeatother

\newcommand*{\KhCubeScaled}[1]{\mathcal{I}\left(#1\right)}
\newcommand*{\KhCubeSimple}[1]{\mathcal{I}(#1)}
\newcommand*{\KhCubeSignedScaled}[2]{\mathcal{I}^{#2}\left(#1\right)}
\newcommand*{\KhCubeSignedSimple}[2]{\mathcal{I}^{#2}(#1)}
\newcommand*{\KhBracketScaled}[1]{\left\llbracket#1\right\rrbracket}
\newcommand*{\KhBracketSimple}[1]{\llbracket#1\rrbracket}

\def\defVersions#1#2{%
	\expandafter\def\csname #1\endcsname{#2}%
	\expandafter\def\csname E#1\endcsname{#2^{ev}}%
	\expandafter\def\csname O#1\endcsname{#2^{odd}}%
	\expandafter\def\csname G#1\endcsname{#2^{cov}}%
}

\newcommand*{\KhCom}{K\!h}

\newcommand*{\Kh}{\mathcal{H}}
\newcommand*{\EKh}{\Kh_{ev}}
\newcommand*{\OKh}{\Kh_{odd}}
\newcommand*{\GKh}{\Kh_{\pi}}

\def\chdeg{\mathrm{deg}\mskip\thinmuskip}
\def\sdeg{\mathrm{sdeg}\mskip\thinmuskip}
\def\wdeg#1{\|#1\|}

\def\sdegCob[#1){\sdeg\left(\textcobordism[#1)\right)}
\def\sdegNum#1#2{\left[#1 \atop #2\right]}
\def\sdegNumTwice#1{\left[#1 \atop #1\right]}

\def\dsdegNum#1#2{\displaystyle{\left[#1 \atop #2\right]}}
\def\dsdegNumTwice#1{\displaystyle{\left[#1 \atop #1\right]}}

\def\permMM{X}
\def\permSS{Y}
\def\permMS{Z}
\def\permSM{Z^{-1}}

\makeatletter

\def\arXiv{\@ifstar\arXiv@@\arXiv@}
\def\arXiv@#1{\href{http://front.math.ucdavis.edu/#1}{arXiv:#1}}
\def\arXiv@@#1#2{\href{http://front.math.ucdavis.edu/#2}{arXiv:#1}}

\makeatother

\allowdisplaybreaks

\psset{unit=7mm}

\selectfont

\begin{document}

\title[Mirror links have dual odd and generalized Khovanov homology]{Mirror links have dual odd and generalized\\Khovanov homology}
\author{Wojciech Lubawski}
\author{Krzysztof K. Putyra}

\date\today

\begin{abstract}
	We show that the~generalized Khovanov homology, defined by the~second author in the~framework of chronological cobordisms, admits a~grading by the~group $\Z\times\Z_2$, in which all homogeneous summands are isomorphic to the~unified Khovanov homology defined over the~ring $\Zpi:=\ZpiLong$ (here, setting $\pi$ to $\pm 1$ results either in even or odd Khovanov homology). The~generalized homology has $\scalars:=\scalarsLong*$ as coefficients, and the~above implies that most of automorphisms of $\scalars$ fix the~isomorphism class of the~generalized homology regarded as $\scalars$-modules, so that the~even and odd Khovanov homology are the~only two specializations of the~invariant. In particular, switching $\permMM$ with $\permSS$ induces a~derived isomorphism between the~generalized Khovanov homology of a~link $L$ with its dual version, i.e.\ the~homology of the~mirror image $L^!$, and we compute an~explicit formula for this map. When specialized to integers it descends to a~duality isomorphism for odd Khovanov homology, which was conjectured by A.~Shumakovitch.
\end{abstract}

\maketitle

\section{Introduction}\label{sec:intro}
In his seminal paper \cite{KhHom} Khovanov constructed for every link $L$ a~sequence of graded abelian groups $\EKh^i(L)$ called the~\emph{Khovanov homology} of the~link $L$. The~graded Euler characteristic of $\EKh(L)$ is the~famous Jones polynomial $J_L(q)$, of which many properties have an~interpretation at the~higher level of homology groups. For instance, for a~mirror link $L^!$ we have $J_{L^!}(q) = J_L(q^{-1})$, which corresponds to duality of Khovanov homology in the~derived sense, i.e.\ there is an~isomorphism $C(L^!)\cong C(L)^*:=\Hom(C(L);\Z)$ between complexes of free groups that compute the~homology.\footnote{
	One can regard this as an~isomorphism between the~Khovanov cohomology of a~link $L$ and the~Khovanov homology of the~mirror link $L^!$.}
Such an~isomorphism was explicitly constructed already in \cite{KhHom}, but its existence can be also deduced from an~extension of the~homology to link cobordisms \cite{KhLinkCobs,DrorCobs}.

The~Khovanov homology is not the~only categorification of the~Jones polynomial. A~distinct homology theory $\OKh(L)$ was discovered by Ozsv\'ath, Rasmussen and Szab\'o \cite{OddKh}, which they called the~\emph{odd Khovanov homology}. Thence, we shall refer to the~original construction as \emph{even}. Both theories agree when regarded with coefficients in the~two-element field $\mathbb{F}_2$, but they are totally different over integers --- there are pairs of knots with the~same homology of one type but different of the~other \cite{ShumComp}. Computer-based calculation revealed the~duality phenomenon for odd homology, but the~theoretical explanation was missing: neither an~extension to link cobordisms of the~odd theory nor an~explicit isomorphism between complexes was known.

\wrapfigure[r]<3>{%
	\psset{unit=1cm}%
	\begin{pspicture}(-2.4,-1ex)(2.4,1.9)
		\diagnode cov( 0,1.55)[\Kh(L)]
		\diagnode  ev(-1.7,0)[\EKh(L)]
		\diagnode odd( 1.7,0)[\OKh(L)]
		\diagarrow|b{npos=0.6,labelsep=3pt}|[cov`ev;\permMM,\permSS,\permMS\,\mapsto<0.7em>\mathrlap1]
		\diagarrow|a{npos=0.65,labelsep=3pt}|[cov`odd;\substack{%
					\ \mathllap{\permMM,\permMS\,}\mapsto<0.7em> 1\phantom{-}\\
					\ \mathllap{\permSS\,}\mapsto<0.7em>-1}]
	\end{pspicture}}%
Both theories were unified by one of the~authors \cite{KhUnified,ChCob}. The~\emph{generalized Khovanov homology} $\Kh(L)$ is a~sequence of modules over the~ring $\scalars:=\scalarsLong$, and both even and odd homology can be recovered by specifying the~generators $\permMM$, $\permSS$, and $\permMS$ into integers. More precisely, the~even and odd homology are isomorphic to $\Kh(L;\Z)$ for various $\scalars$-module structures on $\Z$ (see the~diagram to the~right). This leads into eight possible homology theories, half of which were easily shown to be redundant: multiplying the~action on $\Z$ of each generator $\permMM$, $\permSS$ and $\permMS$ by $-1$ does not change $\Kh(L;\Z)$ \cite{KhUnified}. However, four choices are left and we prove in this paper that they can be further reduced to just two cases: the~even and odd homology. In particular, the~most general theory in our framework, $\GKh(L)$, is defined over the~ring $\Zpi := \ZpiLong$. We prove it by introducing a~new grading by $\Z_2\times\Z$, in which the~generators $\permMM$, $\permSS$ and $\permMS$ has nontrivial degrees. We call it a~\emph{splitting degree}, because it decomposes the~generalized Khovanov homology $\Kh(L)$ into a~bunch of copies of $\GKh(L)$. An~interesting feature of this degree is that it is not multiplicative with respect to the~tensor product.

We use the~new grading to construct an~explicit duality isomorphism at the~end of this paper. The~main difficulty is that, after dualizing the~complex, the~roles of parameters $\permMM$ and $\permSS$ are interchanged. In particular, the~Frobenius-like algebra $A=\scalars v_+\oplus\scalars v_-$, associated to a~circle, is different from its dual one. We overcome this with a~help from the~splitting degree: not only $v_+$ and $v_-$ have different degrees, but also all generators of $A\otimes A$. Roughly speaking, we define a~family of isomorphisms $A^{\otimes k}\to^\cong (A^*)^{\otimes k}$ that intertwines the~algebra and coalgebra operations.

It is worth to notice that the~generalized Khovanov homology $\Kh(L)$ is conjectured to extend projectively to link cobordisms \cite{ChCob}. The~word `projective' means the~assignment
\begin{displaymath}
	\Big\{\textnormal{link cobordisms}\Big\}\to\Big\{\textnormal{chain maps}\Big\}
\end{displaymath}
is defined only up to global invertible scalars. This would be enough to show that $\Kh(L)$ possesses the~duality property similar to the~one for $\EKh(L)$ and, in particular, this would show indirectly the~duality of odd Khovanov homology.

\subsection*{Organization of the~paper}
We first describe briefly the~construction of the~generalized Khovanov homology, including a~discussion on chronological cobordisms and chronological TQFTs. The~splitting degree is defined in Section~\ref{sec:s-degree} for both chronological cobordisms and modules over $\scalars$. In Section~\ref{sec:invariance} we refine the~generalized Khovanov complex to a~graded complex, proving its invariance under Reidemeister moves. The~last section contains the~main results of this paper: the~decomposition of the~generalized Khovanov homology $\Kh(L)$ into a~bunch of copies of the~unifying one $\Kh_\pi(L)$, and the~duality isomorphism between $\Kh(L^!)$ and $\Kh(L)^*$ as well as for the~unifying and odd Khovanov homologies.

\section{Basic definitions}

\subsection{Chronological cobordisms}\label{sec:cobs}
\begin{definition}
	Let $W$ be a~cobordism with a~Riemann metric. A~\emph{chronology} on $W$ consists of a~Morse function $h\colon W\to I$ that separates critical points, and a~choice of an~orientation of $E^-(p)$, the~space of unstable directions in the~gradient flow induced by $h$, at each critical point $p$. We require $h^{-1}(0)$ and $h^{-1}(1)$ to be the~input and output of $W$ respectively.
\end{definition}

Chronological cobordisms admit two disjoint unions: the~`left-then-right' $W\ldsum W'$ and the~`right-then-left' one $W\rdsum W'$. Both are diffeomorphic to the~standard disjoint union $W\sqcup W'$, but to avoid a~situation with two critical points at the~same level, one has to pull all critical points of $W$ below $\frac{1}{2}$ and those of $W'$ over $\frac{1}{2}$ (for $\ldsum$) or the~other way (for $\rdsum$):
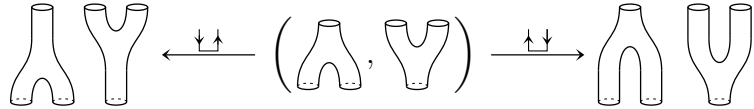
\begin{displaymath}
	\textcobordism[2](M)(sI)\textcobordism[1](sI)(S)
		\from<3em>^{\ldsum}
	\left(\textcobordism[2](M), \textcobordism[1](S)\right)
		\to<3em>^{\rdsum}
	\textcobordism[2](sI)(M)\hskip 0.3\psxunit\textcobordism[1](S)(sI)
\end{displaymath}
A~standard argument from Morse theory shows that every 2-dimensional chronological cobordism can be built from six surfaces:
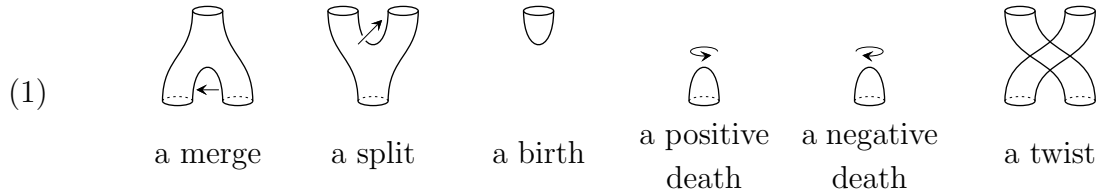
\begin{equation}\label{diag:mor-gens}\psset{unit=1cm}\begin{centerpict}(12.4,2.4)
	\COBmergeFrLeft(0,1.1)\COBsplitFrBack(2.6,1.1)
	\COBbirth(5.2,1.1)\COBposDeath(7.0,1.1)\COBnegDeath(9.2,1.1)
	\COBpermutation(11.2,1.1)
	\rput[B](0.6,1.5ex){\textnormal{a~merge}}
	\rput[B](2.8,1.5ex){\textnormal{a~split}}
	\rput[B](5.0,1.5ex){\textnormal{a~birth}}
	\rput[B](7.2,1.5ex){\parbox{10ex}{\centering\textnormal{a~positive death}}}
	\rput[B](9.4,1.5ex){\parbox{10ex}{\centering\textnormal{a~negative death}}}
	\rput[B](11.8,1.5ex){\textnormal{a~twist}}
\end{centerpict}\end{equation}
The~little arrows visualize orientations of critical points. One merge and one split is sufficient, as the~little arrow can be reversed by composing the~cobordism with the~twist.

\begin{definition}
	Define the~\emph{chronological degree} $\chdeg W\in \Z\times\Z$ of a~chronological cobordism $W$ by setting
	\begin{equation}
		\chdeg W = (\#\text{births}-\#\text{merges}, \#\text{deaths}-\#\text{splits}).
	\end{equation}
\end{definition}

\noindent
The~chronological degree is clearly additive with respect to composition of chronological cobordisms as well as to any of the~disjoint sums.

\begin{lemma}\label{lem:chdeg-vs-bdry}
	Given a~chronological cobordism $W$ of degree $\chdeg W = (a,b)$ with $n$ inputs and $m$ outputs,
	$a+n = b+m$.
\end{lemma}
\begin{proof}
	Straightforward, by checking for generating cobordisms \eqref{diag:mor-gens}.
\end{proof}

Choose a~ring $\scalars:=\scalarsLong$ and let $\kChCob$ be a~$\scalars$-linear category with finite disjoint unions of circles as objects, and $\scalars$-linear combinations of chronological cobordisms as morphisms, modulo the~following \emph{chronological relations}:
\begin{gather}
	%
		\label{rel:reverse-orientation}
		\textcobordism[2](M-R) = \permMM\textcobordism[2](M-L)
			\hskip 1cm
		\textcobordism[1](S-F) = \permSS\textcobordism[1](S-B)
			\hskip 1cm
		\textcobordism[1](sD-) = \permSS\textcobordism[1](sD+)
	\\
	%
		\textcobordism[1](sI)(I) = \textcobordism*[1](slB)(M-L)
			\hskip 1.5cm
		\textcobordism[1](I)(sI) = \textcobordism*[1](S-B)(srD+,2)
			\hskip 1.5cm
		\textcobordism[1](I)(sI) = \textcobordism*[1](S-B)(slD-,1)
	\\
	%
		\textcobordism*[3](M-L)(M-L) = \permMM\textcobordism*[3](M-L,2)(M-L)
			\hskip 1.5cm
		\textcobordism*[1](S-B)(S-B) = \permSS\textcobordism*[1](S-B)(S-B,2)
	\\
	%
		\textcobordism*[2](S-B)(M-L,2) = \permMS\textcobordism*[2](M-L)(S-B)
															 = \textcobordism*[2](S-B,2)(M-L)
	\\
	%
		\label{rel:disjoint-union}
		\begin{centerpict}(-0.1,-0.1)(3.5,2.5)
			\COBcylinder(0.1,0)(0.1,2.4)\COBcylinder(1.0,0)(1.0,2.4)
			\COBcylinder(2.0,0)(2.0,2.4)\COBcylinder(2.9,0)(2.9,2.4)
			\rput[c](0.77,0.1){$\scriptstyle\cdots$}\rput[c](0.77,2.25){$\scriptstyle\cdots$}
			\rput[c](2.67,0.1){$\scriptstyle\cdots$}\rput[c](2.67,2.25){$\scriptstyle\cdots$}
			\psframe[framearc=0.5,fillstyle=solid](-0.05,0.3)(1.55,1.2)
			\psframe[framearc=0.5,fillstyle=solid]( 1.85,1.2)(3.45,2.1)
			\rput[c](0.8,0.75){$W'$}
			\rput[c](2.7,1.65){$W\phantom'$}
		\end{centerpict}
			= \lambda(\chdeg W,\chdeg W')
		\begin{centerpict}(-0.1,-0.1)(3.5,2.5)
			\COBcylinder(0.1,0)(0.1,2.4)\COBcylinder(1.0,0)(1.0,2.4)
			\COBcylinder(2.0,0)(2.0,2.4)\COBcylinder(2.9,0)(2.9,2.4)
			\rput[c](0.77,0.1){$\scriptstyle\cdots$}\rput[c](0.77,2.25){$\scriptstyle\cdots$}
			\rput[c](2.67,0.1){$\scriptstyle\cdots$}\rput[c](2.67,2.25){$\scriptstyle\cdots$}
			\psframe[framearc=0.5,fillstyle=solid](-0.05,1.2)(1.55,2.1)
			\psframe[framearc=0.5,fillstyle=solid]( 1.85,0.3)(3.45,1.2)
			\rput[c](0.8,1.65){$W'$}
			\rput[c](2.7,0.75){$W\phantom'$}
		\end{centerpict}
\end{gather}
where $W$ and $W'$ are any cobordisms and $\lambda(a,b,a',b') = \permMM^{aa'}\permSS^{bb'}\permMS^{ab'-a'b}$. We proved in \cite{ChCob} the~following non-degeneracy result for $\kChCob$.

\begin{proposition}\label{prop:kChCob-nondeg}
	Suppose $kW=0$ for a~chronological cobordism $W$ and a~nonzero $k\in\scalars$. Then $W$ has either positive genus or at least two closed components, and $k$ is divisible by $(\permMM\permSS-1)$. In particular, cobordisms cannot be annihilated by monomials.
\end{proposition}

\begin{remark}\label{rmk:other-coeffs-for-chcob}
	Given a~ring homomorphism $\scalars\to<1em> R$ we define $R\ChCob$ likewise. In particular, if we consider $\Z$ as a~trivial $\scalars$-module, i.e.\ $\permMM$, $\permSS$, and $\permMS$ act as the~identity, $\Z\ChCob$ is the~linear extension of ordinary cobordisms: the~relations \eqref{rel:reverse-orientation}--\eqref{rel:disjoint-union} become equalities.
\end{remark}

\subsection{The~generalized Khovanov complex}\label{sec:khov-def}
We shall now briefly describe the~construction of the~generalized Khovanov complex. We encourage the~reader to refer to Fig.~\ref{fig:trefoil-cube} frequently while reading this section; it illustrates the~construction for the~right-handed trefoil.

Fix a~link diagram $D$ and enumerate its crossings. Given a~sequence $\xi=(\xi_1,\dots,\xi_n)$, where $\xi_i\in\{0,1\}$ and $n$ is the~number of crossings in $D$, let $D_\xi$ be a~collection of circles obtained by resolving each crossing as illustrated below.
\begin{displaymath}
		\psset{unit=5mm}
		\begin{centerpict}(-1,-1)(1,1)
			\psbezier(-1,-1)(0,-0.1)(0,-0.1)(1,-1)
			\psbezier(-1, 1)(0, 0.1)(0, 0.1)(1, 1)
		\end{centerpict}
			\quad\from^{\xi_i=0}\quad
		\begin{centerpict}(-1,-1)(1,1)
			\psline(-1,-1)(1,1)
			\psline[border=5\pslinewidth](-1,1)(1,-1)
			\rput[b](0,1.3ex){$\scriptstyle i$}
		\end{centerpict}
			\quad\to^{\xi_i=1}\quad
		\begin{centerpict}(-1,-1)(1,1)
			\psbezier(-1,-1)(-0.1,0)(-0.1,0)(-1,1)
			\psbezier( 1,-1)( 0.1,0)( 0.1,0)( 1,1)
		\end{centerpict}
\end{displaymath}
We call them type 0 and type 1 resolutions of a~crossing. The~diagrams $D_\xi$ decorate vertices of an~$n$-dimensional cube $\KhCube{D}$, called the~\emph{cube of resolutions} of $D$. Let $|\xi|:=\xi_1+\ldots+\xi_n$ be the~\emph{weight} of the~vertex $\xi$. An~edge $\zeta\colon\xi\to\xi'$, oriented towards the~vertex with higher weight, is decorated with a~cobordism $D_\zeta\subset\R^2\times I$ that is a~vertical surface except a~small neighborhood of the~resolution being changed, where a~saddle \drawSaddle\ is inserted.\footnote{
	In Fig.~\ref{fig:trefoil-cube} we use the~surgery description of cobordisms: the~input circles together with an~arc, a~surgery along which results in the~output circles. The~arc is oriented, inducing an~orientation of the~saddle. A~3D picture of one cobordism is provided in the~left bottom corner of the~picture.
} Decorate each crossing of $D$ with a~small arrow that connects the~two arcs in type 0 resolution---these arrows determine uniquely orientations of saddle points of the~cobordisms $D_\zeta$, so that $\KhCube{D}$ can be regarded as a~diagram in the~category $\kChCob$.

\begin{figure}
	\begin{center}\input{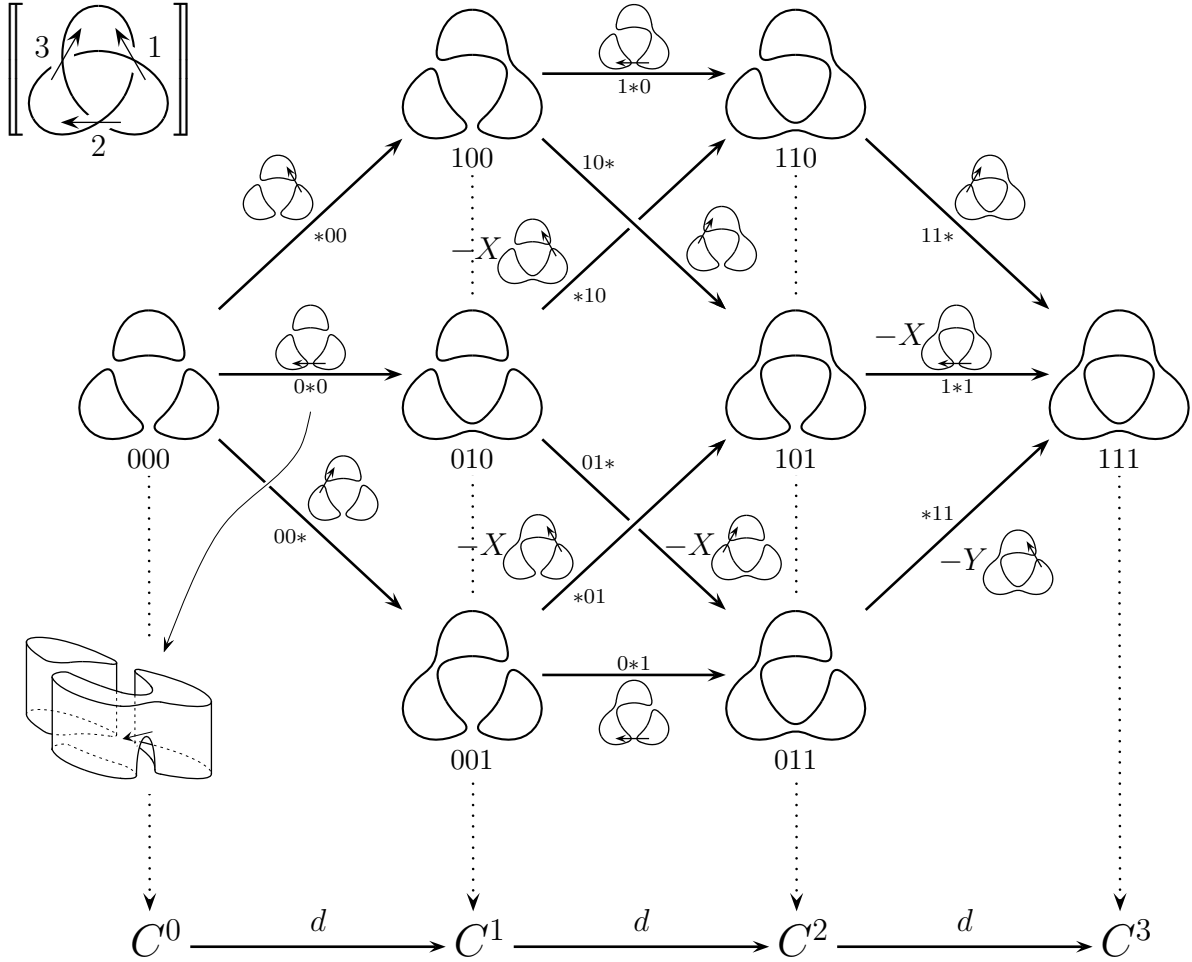}\end{center}
	\caption{The~cube of resolutions and the~generalized Khovanov complex for the~right-handed trefoil.}
	\label{fig:trefoil-cube}
\end{figure}

The~cube $\KhCube{D}$ does not commute in general, but there is a~cubical cocycle $\psi\in C^2(I^n;\invScalars)$ such that for every face $S$ of the~form
\begin{equation}
	\begin{diagps}(-2cm,-1.1cm)(2cm,1.4cm)
		\node 00(-1.5cm, 0cm)[D_{00}]
		\node 01( 0.0cm,-1cm)[D_{01}]
		\node 10( 0.0cm, 1cm)[D_{10}]
		\node 11( 1.5cm, 0cm)[D_{11}]
		\arrow|a|{->}[00`10;W_{{\star}0}]
		\arrow|a|{->}[10`11;W_{1{\star}}]
		\arrow|b|{->}[00`01;W_{0{\star}}]
		\arrow|b|{->}[01`11;W_{{\star}1}]
	\end{diagps}
\end{equation}
the~twisted commutativity $W_{1{\star}}W_{{\star}0} = \psi(S)W_{{\star}1}W_{0{\star}}$ holds. If $\delta\epsilon=-\psi$ for a~cubical cochain $\epsilon\in C^1(I^n;\invScalars)$, the~corrected cube $\KhCubeSigned{D}\epsilon$, in which each cobordism $D_\zeta$ is multiplied by $\epsilon(\zeta)$, anticommutes. We call such a~1-cochain a~\emph{sign assignment}.

The~generalized Khovanov complex is constructed in the~\emph{additive closure} $\catAdd{\kChCob}$ of the~category $\kChCob$: objects are finite sequence (vectors) of $1$-manifolds, and morphisms are matrices with linear combinations of chronological cobordisms as its entries. A~direct sum in $\catAdd{\kChCob}$ is realized by concatenation of sequences.

\begin{definition}
	Let $D$ be a~link diagram with enumerated and oriented crossings. Given a~sign assignment $\epsilon$ for the~cube $\KhCube{D}$ we define the~\emph{generalized Khovanov bracket} as the~chain complex $\KhBracket{D}_\epsilon$ in the~category $\catAdd{\kChCob}$ with
	\begin{equation}\label{eq:bracket-def}
		\KhBracket{D}_\epsilon^i := \bigoplus_{|\xi|=i} D_\xi,
			\hskip 1cm
		d^i|_{D_\xi} := \sum_{\mathclap{\zeta\colon\xi\to<0.7em>\xi'}} \epsilon(\zeta)D_\zeta.
	\end{equation}
	The~\emph{generalized Khovanov complex} $\KhCom(D)$ is obtained from $\KhBracket{D}_\epsilon$ by shifting it to the~left by the~number of negative crossings, i.e.\ $\KhCom(D)^i := \KhBracket{D}_\epsilon^{i+n_-}$.
\end{definition}

\begin{remark}
	The~generalized Khovanov bracket and complex admits an~integral grading induced from the~$\Z{\times}\Z$-grading of chronological cobordisms, see \cite{ChCob}. We skip the~details, as this degree does not play any role in this paper.
\end{remark}

\begin{theorem}[cf.\ \cite{ChCob}]\label{thm:invariance}
	The~homotopy type of the~generalized Khovanov complex $\KhCom(D)$ is a~link invariant, when regarded as a~complex in the~category $\catAdd{\kChCob}$ modulo the~following three relations
	\begin{align*}
		\textnormal{(\textit{S})}&\psset{unit=1cm}\quad\pictRelS = 0
		\hskip 2cm
		\textnormal{(\textit{T})}\quad\pictRelT = \permMS(\permMM+\permSS)\\[1ex]
		\textnormal{(\textit{4Tu})}&\quad\psset{unit=0.75cm}%
			\permMS\pictRelTuL + \permMS\pictRelTuR = \permMM\pictRelTuB + \permSS\pictRelTuT
	\end{align*}
	in which all deaths are oriented clockwise.
\end{theorem}

\noindent
The~proof of the~theorem can be found in \cite{ChCob}. It is revisited in Section~\ref{sec:invariance}, where we inspect the~chain maps involved in the~proof against the~new grading described in the~next section.

\subsection{Chronological TQFTs and homology}\label{sec:tqft}
Consider the~category $\Mod\scalars$ of $\scalars$-modules graded by the~group $\Z\times\Z$. We redefine the~tensor product for homomorphisms by setting for homogeneous maps $f$ and $g$
\begin{equation}\label{eq:graded-tensor-morphisms}
		(f\otimes g)(m\otimes n) := \lambda(\deg g, \deg m) f(m)\otimes g(n),
\end{equation}
where $\lambda(a,b,a',b')=\permMM^{aa'}\permSS^{bb'}\permMS^{ab'-a'b}$ is defined as for $\kChCob$. One checks directly that
\begin{equation}\label{eq:graded-tensor-functoriality}
	(f'\otimes g')\circ(f\otimes g) = \lambda(\deg g',\deg f)(f'\circ f)\otimes(g'\circ g).
\end{equation}
Hence, $\Mod\scalars$ is a~\emph{graded tensor category} in the~sense of \cite{ChCob}. There is a~symmetry $\tau_{M,N}\colon M\otimes N\to N\otimes M$ given by the~formula $\tau_{M,N}(m\otimes n) = \lambda(\deg m,\deg n)\,n\otimes m$ for homogeneous elements $m\in M$ and $n\in N$.

\begin{definition}
	A~\emph{chronological TQFT} is a~functor $\F\colon\kChCob \to \Mod\scalars$ that preserves the~$\Z\times\Z$ grading, and which maps the~`right-then-left' disjoint union $\rdsum$ into the~graded tensor product $\otimes$ and the~twist {\psset{unit=4mm}\textcobordism[2](P)} into the~symmetry $\tau$.
\end{definition}

\noindent
We defined in \cite{ChCob} a~chronological TQFT $\FA\colon\scalars\ChCob \to \Mod\scalars$ that maps a~circle to the~module $A$ freely degenerated by $v_+$ in degree $(1,0)$ and $v_-$ in degree $(0,-1)$, and generating cobordisms to the~following maps:
\begin{align}
	\label{eq:F-merge}
		\FA\left(\textcobordism[2](M-L)\right)&\colon A\otimes A\to A,\phantom{R}\quad
			\begin{cases}
					v_+\otimes v_+ \mapsto v_+, &\qquad v_+\otimes v_- \mapsto v_-,\\
					v_-\otimes v_- \mapsto 0,   &\qquad v_-\otimes v_+ \mapsto \permMM\permMS v_-,
				\end{cases} \\
	\label{eq:F-split}
		\FA\left(\textcobordism[1](S-B)\right)&\colon A\to A\otimes A,\phantom{R}\quad
			\begin{cases}
				v_+\mapsto v_-\otimes v_+ + \permSS\permMS v_+\otimes v_-,\\
				v_-\mapsto v_-\otimes v_-,
			\end{cases} \\
	\label{eq:F-birth}
		\FA\left(\textcobordism[0](sB)\right)&\colon\scalars\to A,\phantom{A\otimes A}\quad
			\begin{cases}
				1\mapsto v_+,
			\end{cases}\\
	\label{eq:F-death}
		\FA\left(\textcobordism[1](sD-)\right)&\colon A\to\scalars,\phantom{A\otimes A}\quad
			\begin{cases}
				v_+\mapsto 0,\\
				v_-\mapsto 1.
			\end{cases}
\end{align}
It is easy to see that $\FA$ preserves the~$\Z\times\Z$-grading. Likewise, given a~ring homomorphism $\scalars\to<1em> R$, we define a~TQFT $\F_{\!R}\colon R\ChCob \to \Mod{R}$ such that $\F_{\!R}(\fntCircle) = A\otimes R$.

\begin{definition}
	The~\emph{generalized Khovanov homology} $\Kh(L)$ of a~link $L$ is the~homology of the~chain complex $\FA\KhCom(D)$, where $D$ is a~diagram of $L$. Given a~$\scalars$-module $M$ we write $\Kh(L;M)$ for the~homology of the~complex $\FA\KhCom(D)\otimes M$.
\end{definition}


\begin{example}
	We distinguish two out of eight $\scalars$-algebra structures on the~ring $\Z$:
	\begin{itemize}
		\item $\Zev$, on which all $\permMM$, $\permSS$, and $\permMS$ acts trivially, and
		\item $\Zodd$, on which $\permMM$ and $\permMS$ acts trivially, but $\permSS$ acts as $-1$.
	\end{itemize}
	It was shown in \cite{ChCob} that $\Kh(L;\Zev)$ and $\Kh(L;\Zodd)$ are the~even and odd Khovanov homology respectively. Both are specializations of $\Kh(L;\Zpi)$, where $\Zpi=\ZpiLong$ is a~$\scalars$-module, on which both $\permMM$ and $\permMS$ act as identity, and $\permSS$ as multiplication by $\pi$.
\end{example}

\section{The~new grading}\label{sec:s-degree}

\subsection{The~grading on chronological cobordisms}\label{sec:gradings}
The~category $\kChCob$ of chronological cobordisms admits an~additional grading by the~group $\Z_2\times\Z$, which we shall refer to as the~\emph{splitting degree} for a~reason explained later. It takes the~following values on the~generating cobordisms:
\begin{align}
	\label{eq:split-degree-merge}
	\sdegCob[2](M-L)	&= \sdegNum{0}{0},					&		\sdegCob[0](sB)		&= \sdegNum{0}{0},	\\
	\label{eq:split-degree-split}
	\sdegCob[1](S-B)	&= \sdegNum{\phantom-0}{-2},&		\sdegCob[1](sD+)	&= \sdegNum{0}{1},	\\
	\label{eq:split-degree-other}
	\sdegCob[2](P)		&= \sdegNum{1}{0},					&		\sdegCob[1](sD-)	&= \sdegNum{1}{1},
\end{align}
and the~coefficient ring $\scalars$ is graded itself with $\sdeg\permMM = \sdeg\permSS=\sdegNum{1}{0}$ and $\sdeg\permMS=\sdegNum{\phantom-0}{-1}$. We use the~vertical notation for elements of $\Z_2\times\Z$ to distinguish it from the~chronological degree. This degree is not additive with respect to the~disjoint union; instead we set
\begin{equation}\label{eq:split-degree-disjoint-union}
	\def\COBxsize{0.3}
	\sdeg\bigg(
		\underbrace{\textcobordism[1](I){\scriptstyle\cdots}\textcobordism[1](I)}_k
		\begin{centerpict}(1.2,-0.1)(2.7,1.3)
			\COBcylinder(1.4,0)(1.4,1.2)\COBcylinder(2.2,0)(2.2,1.2)
			\rput[c](1.95,0.1){$\scriptstyle\cdots$}
			\rput[c](1.95,1.05){$\scriptstyle\cdots$}
			\psframe[framearc=0.5,fillstyle=solid](1.25,0.3)(2.65,0.9)
			\rput[c](1.95,0.6){$\scriptstyle W$}
		\end{centerpict}
		\underbrace{\textcobordism[1](I){\scriptstyle\cdots}\textcobordism[1](I)}_\ell
	\bigg)%
	:= \sdeg W + \sdegNum{k\alpha+\ell\beta}{(k+\ell)\beta},
\end{equation}
assuming $\chdeg W=(\alpha,\beta)$. The~above formula is clearly additive with respect to composition of cobordisms, and it is coherent with the~symmetry:

\begin{equation}\label{eq:split-degree-twist}
	\def\COBxsize{0.3}
	\begin{split}
		\raisebox{0pt}[17mm][7mm]{$\displaystyle{%
		\sdeg\left(\begin{centerpict}(-0.1,0)(1.8,2)
			\cobordism{\COBcylinder(0,0)(0,1)\COBcylinder(0.8,0)(0.8,0.9)\COBcylinder(1.4,0)(1.4,1)}{}%
				{\COBcylinder(0,1)(0.6,2)\COBcylinder(0.8,0.9)(1.4,2)\COBcylinder(1.4,1)(0,2)}
			\rput[c](0.55,0.1){$\scriptstyle\cdots$}
			\rput[c](1.15,1.85){$\scriptstyle\cdots$}
			\psframe[fillstyle=solid,framearc=0.5](-0.2,0.3)(1.3,1)
			\rput[c](0.55,0.65){$\scriptstyle W$}
			\rput[t](0.55,-0.1){$\underbrace{\hskip 1.2\psxunit}_a$}%
			\rput[b](1.15,2.1){$\overbrace{\hskip 1.2\psxunit}^b$}%
		\end{centerpict}\right)}$}
		= \sdeg W + \sdegNum{\beta+b}{\beta} &\\
		= \sdeg W + \sdegNum{\alpha+a}{\beta} &=
		\raisebox{0pt}[9mm][14mm]{$\displaystyle{%
		\sdeg\left(\begin{centerpict}(-0.1,0)(1.8,2)
			\cobordism{\COBcylinder(0,0)(0.6,1.1)\COBcylinder(0.8,0)(1.4,1)\COBcylinder(1.4,0)(0,1)}{}
				{\COBcylinder(0.6,1.1)(0.6,2)\COBcylinder(1.4,1)(1.4,2)\COBcylinder(0,1)(0,2)}
			\rput[c](0.55,0.1){$\scriptstyle\cdots$}
			\rput[c](1.15,1.85){$\scriptstyle\cdots$}
			\psframe[fillstyle=solid,framearc=0.5](0.4,1)(1.9,1.7)
			\rput[c](1.15,1.35){$\scriptstyle W$}
			\rput[t](0.55,-0.1){$\underbrace{\hskip 1.2\psxunit}_a$}%
			\rput[b](1.15,2.1){$\overbrace{\hskip 1.2\psxunit}^b$}%
		\end{centerpict}\right)}$}
	\end{split}
\end{equation}
where $\chdeg W = (\alpha,\beta)$, and the~equality $\alpha+a=\beta+b$ follows from Lemma~\ref{lem:chdeg-vs-bdry}.

\begin{proposition}\label{prop:sdeg-vs-chronology}
	The~splitting degree is coherent with chronological relations.
\end{proposition}
\begin{proof}
	\psset{unit=6mm}
	Creation and annihilation do not change the~degree, as we can directly compute
	\begin{equation}
		\sdeg\left(\textcobordism*[1](slB)(M-L)\right) =
		\sdeg\left(\textcobordism*[1](S-B)(slD-)\right) =
		\sdeg\left(\textcobordism*[1](S-B)(srD+,2)\right) =
		\sdegNum 00.
	\end{equation}
	Choose cobordisms $W_i\colon a_i\bS\to b_i\bS$ for $i=1,2$. If $\chdeg W_i = (m_i,s_i)$, we have
	\begin{align}
		\sdeg(W_1\rdsum W_2) &= \sdeg(W_1\sqcup C_{b_2\bS}) + \sdeg(C_{a_1\bS}\sqcup W_2),\\
		\sdeg(W_1\ldsum W_2) &= \sdeg(W_1\sqcup C_{a_2\bS}) + \sdeg(C_{b_1\bS}\sqcup W_2),
	\end{align}
	where $C_{n\bS}$ is a~disjoint union of $n$ vertical tubes. Using the~formula \eqref{eq:split-degree-disjoint-union} we compute
	\begin{multline}
		\sdeg(W_1\rdsum W_2) - \sdeg(W_1\ldsum W_2)
				=\sdegNum{(b_2-a_2)s_1}{(b_2-a_2)s_1)} + \sdegNum{(a_1-b_1)m_2}{(a_1-b_1)s_2)}\\[1ex]
				=\sdegNum{s_1s_2 + m_1m_2}{s_1m_2 - m_1s_2}
				=\sdeg\left(\permMM^{m_1m_2}\permSS^{s_1s_2}\permMS^{m_1s_2-s_1m_2}\right),
	\end{multline}
	which shows that $W_1\rdsum W_2$ and $\lambda(\chdeg W_1,\chdeg W_2)W_1\ldsum W_2$ have the~same degree. The~remaining chronological relations are easily checked by hand:
	
	\begin{gather}
		\sdeg\left(\textcobordism*[3](M-L,2)(M-L)\right)
				= \sdegNum 1 0
				= \sdeg\left(\permMM\textcobordism*[3](M-L)(M-L)\right),
		\\
		\sdeg\left(\textcobordism*[1](S-B)(S-B,2)\right)
				= \sdegNum{\phantom-0}{-5}
				=\sdeg\left(\permSS\textcobordism*[1](S-B)(S-B)\right),
		\\
		\sdeg\left(\textcobordism*[2](S-B)(M-L,2)\right)
			= \sdeg\left(\permMS\textcobordism*[2](M-L)(S-B)\right)
			= \sdeg\left(\textcobordism*[2](S-B,2)(M-L)\right)
			= \sdegNum{\phantom-0}{-3}.
	\end{gather}
	\vskip-\baselineskip
\end{proof}

\begin{observation}\label{obs:homog-decomp}
	The~splitting degree decomposes the~module of morphisms into
	\begin{equation}\label{eq:mor-graded}
		\Mor(\Sigma,\Sigma') := \bigoplus_{\mathclap{(k,\ell)\in\Z_2\times\Z}} \Mor_{(k,\ell)}(\Sigma,\Sigma'),
	\end{equation}
	where each summand is isomorphic to the~submodule of degree-preserving maps. Indeed, $\sdeg(\permMM^a\permMS^b f)=\sdegNum00$ if $\sdeg f = \sdegNum ab$.
\end{observation}

We make $\scalars\ChCob$ a~graded category by replacing its objects with symbols $\Sigma\{a,b\}$, where $\Sigma$ is an~object of $\scalars\ChCob$ and $(a,b)\in\Z_2\times\Z$. Thinking of $\{a,b\}$ as a~\emph{degree shift} operation, the~module of morphisms is given as a~direct sum
\begin{equation}\label{eq:mor-graded-shifted}
	\Mor(\Sigma\{a,b\},\Sigma'\{a',b'\}) := \bigoplus_{\mathclap{(k,\ell)\in\Z_2\times\Z}}
		\Mor_{(k+a-a',\ell+b-b')}(\Sigma,\Sigma'),
\end{equation}
i.e.\ a~homogeneous morphism $f\in\Mor_{k,\ell}(\Sigma,\Sigma')$, when regarded as $f\colon\Sigma\{a,b\}\to\Sigma'\{a',b'\}$, has degree $\sdeg f = \sdegNum{k-a+a'}{\ell-b+b'}$. We write $\kChCob_0$ for the~subcategory of degree-preserving morphisms.

\subsection{The~grading on modules}\label{sec:decomp}
Recall the~ring $\scalars$ is graded by $\Z_2\times\Z$ with $\sdeg\permMM = \sdeg\permSS = \sdegNum10$ and $\sdeg\permMS=\sdegNum{\phantom-0}{-1}$. Hereafter let $\Mod\scalars$ stand for the~category of modules with a~compatible $\Z_2\times\Z$-grading, in addition to the~$\Z\times\Z$-grading, and we write $\Mod{\scalars,0}$ for the~subcategory formed by maps that preserve the~new degree. Again, the~new grading is not additive with respect to the~tensor product, but instead we set
\begin{equation}\label{eq:split-deg-Frob-tensor}
	\sdeg(m\otimes n) := \sdeg(m) + \sdeg(n) + \sdegNum{\beta\wdeg{n}}{\beta\wdeg{n}}
\end{equation}
for homogeneous $m\in M$ and $n\in N$, where $\deg m = (\alpha,\beta)$ and $\wdeg{n}$ is the~\emph{weight} of $n$: the~difference of the~two components of $\deg$ (e.g.\ $\wdeg m = \alpha-\beta$). The~name is motivated by the~behavior of the~symmetry isomorphism: it is homogeneous only when restricted to submodules supported in a~single weight.

\begin{lemma}\label{lem:sdeg-can-isoms}
	The~associator $(M_1\otimes M_2)\otimes M_3 \to<2.5em> M_1\otimes(M_2\otimes M_3)$ preserves the~splitting degree. Moreover, if $M_1$ and $M_2$ are supported in weights $w_1$ and $w_2$ respectively, the~symmetry $\tau\colon M_1\otimes M_2\to M_2\otimes M_1$ is homogeneous of degree $\sdeg\tau = \sdegNum{w_1w_2}{0}$.
\end{lemma}
\begin{proof}
	Choose elements $m_i\in M_i$, $i=1,2,3$, with $\Z\times\Z$ degrees $\deg m_i=(\alpha_i,\beta_i)$. Using formula \eqref{eq:split-deg-Frob-tensor} we compute
	\begin{align*}
			\sdeg((&m_1\otimes m_2)\otimes m_3)
				= \sdeg(m_1\otimes m_2) + \sdeg(m_3) +
					\sdegNumTwice{(\beta_1+\beta_2)\wdeg{m_3}} \\
				&= \sdeg(m_1) + \sdeg(m_2) + \sdeg(m_3) +
					\sdegNumTwice{\beta_1 \wdeg{m_2} + \beta_1 \wdeg{m_3} + \beta_2 \wdeg{m_3}} \\
				&= \sdeg(m_1) + \sdeg(m_2\otimes m_3) +
					\sdegNumTwice{\beta_1\wdeg{m_2\otimes m_3}}\\
				&= \sdeg(m_1\otimes (m_2\otimes m_3)).
	\end{align*}
	For the~second statement, first compute $\tau(m_1\otimes m_2) = \permMM^{\alpha_1\alpha_2}\permSS^{\beta_1\beta_2}\permMS^{\alpha_1\beta_2-\beta_1\alpha_2}m_2\otimes m_1$. Then
	\begin{displaymath}
		\begin{split}
			\sdeg(\tau(&m_1\otimes m_2)) - \sdeg(m_1\otimes m_2) = \\
				&= \sdegNum{\alpha_1\alpha_2+\beta_1\beta_2}{\beta_1\alpha_2-\alpha_1\beta_2}
					+ \sdeg(m_2\otimes m_1)-\sdeg(m_1\otimes m_2)\\
				&= \sdegNum{\alpha_1\alpha_2+\beta_1\beta_2}{\beta_1\alpha_2-\alpha_1\beta_2}
					+ \sdegNumTwice{\beta_2 w_1} - \sdegNumTwice{\beta_1 w_2}\\
				&= \sdegNum{(\alpha_1-\beta_1)(\alpha_2-\beta_2)}{0}
				 = \sdegNum{w_1w_2\mathrlap{\phantom{\beta}}}{0}.
		\end{split}
	\end{displaymath}
	\vskip-\baselineskip
\end{proof}	

\begin{lemma}
	Choose a~homogeneous map $f\colon M\to N$ of degree $\deg f = (\alpha,\beta)$, and two modules $M'$ and $M''$ supported in weights $k$ and $\ell$ respectively. Then
	\begin{equation}\label{eq:}
		\sdeg(\id_{M'}\otimes f\otimes\id_{M''}) = \sdeg f+\sdegNum{k\alpha+\ell\beta}{(k+\ell)\beta}.
	\end{equation}
	In particular, the~graded tensor product relation \eqref{eq:graded-tensor-functoriality} is graded.
\end{lemma}
\begin{proof}
	Pick homogeneous elements $m_1\in M'$, $m_2\in M$ and $m_3\in M''$, each of $\Z\times\Z$ degree $\deg(m_i)=(\alpha_i,\beta_i)$. Then
	\begin{displaymath}
		\begin{split}
			\sdeg(&(\id\otimes f\otimes\id)(m_1\otimes m_2\otimes m_3)) - \sdeg(m_1\otimes m_2\otimes m_3)\\
			  &= \sdeg\left(\permMM^{\alpha_1\alpha}\permSS^{\beta_1\beta}\permMS^{\alpha_1\beta-\beta_1\alpha}
						m_1\otimes f(m_2)\otimes m_3\right) - \sdeg(m_1\otimes m_2\otimes m_3)\\
				&= \sdeg f + \sdegNum{\beta_1\beta + \alpha_1\alpha}{\beta_1\alpha-\beta\alpha_1}
					+ \sdegNumTwice{\beta_1(\alpha-\beta) + \beta(\alpha_3-\beta_3)}
				= \sdeg f + \sdegNum{k\alpha + \ell\beta}{(k+\ell)\beta}.
		\end{split}
	\end{displaymath}
	The~last statement follows from a~direct computation, as in Proposition~\ref{prop:sdeg-vs-chronology}.
\end{proof}

The~generators of the~module $A$ have weights $\wdeg{v_+} = \wdeg{v_-} = 1$, implying that
\begin{equation}\label{eq:sdeg-tensor-with-A}
	\sdeg\left(\id_{A^{\otimes k}}\otimes f\otimes\id_{A^{\otimes\ell}}\right) = \sdeg f + \sdegNum{k\alpha + \ell\beta}{(k+\ell)\beta},
\end{equation}
which is similar to formula~\eqref{eq:split-degree-disjoint-union}. We define the~splitting degree on $A$ by setting $\sdeg v_+ = \sdegNum{0}{0}$ and $\sdeg v_-=\sdegNum{\phantom-1}{-1}$; Table~\ref{tab:sdeg-for-A2} contains degrees of generators of $A^{\otimes 2}$.

\begin{table}
	\begin{tabular}{r|cccc}
		\hline
		Generator    & $\vv++$ & $\vv+-$  & $\vv-+$  & $\vv--$ \\
		\hline
		$\deg$:
			& $(2,0)$ & $(1,-1)$ & $(1,-1)$ & $(0,-2)\mathrlap{\phantom{\dsdegNumTwice 0}}$\\
		$\sdeg$\rule[-3ex]{0pt}{1pt}:
			& $\dsdegNum{0}{0}$ & $\dsdegNum{\phantom-1}{-1}$ & $\dsdegNum{\phantom-0}{-2}$ & $\dsdegNum{\phantom-1}{-3}$\\
		\hline
	\end{tabular}
	\vskip 0.5\baselineskip
	\caption{Degrees of generators of the~second power of $A$.}\label{tab:sdeg-for-A2}
\end{table}

\begin{lemma}\label{lem:power-degrees}
	A~generator $v = \vvv{k,*\cdots,1}\in A^{\otimes k}$ is homogeneous of degree $\sdeg v = \sdegNum aa$ with
	$a=-\displaystyle{\sum_{\mathclap{v_i=v_-}}i}$.
\end{lemma}
\begin{proof}
	The~lemma follows from an~easy induction argument and is left to the~reader.
\end{proof}

\begin{proposition}\label{prop:FA-is-graded}
	The~functor $\FA\colon\scalars\ChCob\to\Mod\scalars$ preserves the~splitting degree.
\end{proposition}
\begin{proof}
	In the~view of Lemma~\ref{lem:sdeg-can-isoms} and formula \eqref{eq:sdeg-tensor-with-A} it is enough to check that the~four maps \eqref{eq:F-merge}--\eqref{eq:F-death} have the~same degrees as the~corresponding cobordisms. This follows directly from the~expressions for these maps and Lemma~\ref{lem:power-degrees}.
\end{proof}

\section{Invariance revisited}\label{sec:invariance}
\def\proofpart#1{\medskip\noindent\textit{#1.}\space\ignorespaces}%
We shall introduce the~splitting degree to the~generalized Khovanov complex. Choose a~link diagram $D$ and construct its cube of resolutions $\KhCubeSigned{D}{\epsilon}$ corrected by a~certain sign assignment $\epsilon$. For every vertex $\xi$ choose a~directed path to $\xi$ originating at the~initial vertex $(0,\dots,0)$, and denote by $W_\xi$ the~cobordism the~path encodes.\footnote{
	The~path is empty if $\xi=(0,\dots,0)$, in which case $W_\xi$ is the~cylinder $D_\xi\times I$.
}
Shift vertices of the~cube by degrees of the~cobordisms $W_\xi$:
\begin{equation}
	\KhCubeSigned{D}{\epsilon}(\xi) := D_\xi\{\sdeg W_\xi\}.
\end{equation}
Because the~faces anticommute, $\sdeg W_\xi$ does not depend on the~path chosen. This modification results in a~cube in the~category $\kChCob_0$, i.e.\ all morphisms preserve the~splitting degree. In particular, the~differential in the~generalized Khovanov bracket $\KhBracket{D}_\epsilon$ is a~degree-preserving map.

\begin{theorem}\label{thm:inv-graded-complex}
	The~homotopy type of the~graded generalized Khovanov complex is a~link invariant. In particular, the~generalized Khovanov homology $\Kh(L)$ admits a~$\Z_2\times\Z$-grading coherent with the~action of\/ $\scalars$.
\end{theorem}


For Theorem~\ref{thm:inv-graded-complex} to make sense, the~relations \textit{S}, \textit{T} and \textit{4Tu} from Theorem~\ref{thm:invariance} must be homogeneous. This follows from a~direct computation. Our goal is to show that all isomorphisms involved in the~proof of invariance from \cite{ChCob} are homogeneous---this is enough, as any homogeneous isomorphism can be made graded by scaling it with some monomial $\permMM^a\permMS^b$, see Observation~\ref{obs:homog-decomp}. We first show that the~grading does not depend on the~extra choices made in the~construction of the~generalized Khovanov bracket. The~key tool is the~following result.

\begin{lemma}\label{lem:one-comp-is-OK}
	Suppose there is a~commutative square in $\scalars\ChCob$
	\begin{equation}
		\begin{diagps}(0,-1ex)(6em,10ex)
			\square(0,0)<6em,8ex>[\Sigma_0`\Sigma_1`\Sigma'_0`\Sigma'_1;f_0`g`g'`f_1]
		\end{diagps}
	\end{equation}
	where each morphism is a~chronological cobordism scaled by an~invertible element from $\scalars$. If each $f_i$ is graded with respect to the~splitting degree, $\sdeg g=\sdeg g'$.
\end{lemma}
\begin{proof}
	It is enough to show that the~composition $f_1g = g'f_0$ does not vanish. This follows from Proposition~\ref{prop:kChCob-nondeg}.
\end{proof}

\proofpart{Sign assignments}
Given two sign assignments $\epsilon_1$ and $\epsilon_2$ of the~cube $\KhCube{D}$, the~corrected cubes $\KhCubeSigned{D}{\epsilon_1}$ and $\KhCubeSigned{D}{\epsilon_2}$ are isomorphic via a~family of morphisms $f_\xi := \nu(\xi)\id$, where $\nu\in C^0(I^n;\invScalars)$ is a~cochain such that $\epsilon_2=\delta\nu\cdot\epsilon_1$. Hence, each $f_\xi$ is a~homogeneous map.

\proofpart{Arrows over crossings}
Choose link diagrams $D$ and $D'$ that differ only in the~direction of arrows decorating the~crossings. Then $\KhCube{D'}$ is the~cube $\KhCube{D}$ with some edges scaled by $\permMM$ or $\permSS$, due to \eqref{rel:reverse-orientation}. These coefficients define a~cochain $\lambda\in C^1(I^n;\invScalars)$, and if $\epsilon$ is a~sign assignment for $\KhCube{D}$, so is $\epsilon\lambda^{-1}$ for $\KhCube{D'}$. One can easily see that $\KhCubeSigned{D}\epsilon = \KhCubeSigned{D'}{\epsilon\lambda^{-1}}$.

\proofpart{Orderings on crossings and circles}
A~change in enumeration of crossings permutes only the~summands in \eqref{eq:bracket-def}. On the~other hand, each component of the~isomorphism of cubes that reorders circles in resolutions is a~composition of twists. Hence, it is homogeneous, and we again use Lemma~\ref{lem:one-comp-is-OK} to deduce all components have the~same splitting degree.

\begin{corollary}
	The~isomorphism class of the~graded generalized Khovanov bracket $\KhBracket{D}$ depends only on the~link diagram $D$.
\end{corollary}

\noindent
We shall now proceed to Reidemeister moves. Our goal is to show that the~chain homotopy equivalences defined in \cite{ChCob} are homogeneous. We shall recall how they are defined, but a place for the diagram for clarity all cobordisms are drawn without arrows orienting their critical points. The~convention to keep in mind is that deaths are oriented clockwise, whereas arrows orienting merges and splits point towards right or front.

\proofpart{Reidemeister I}
\wrapfigure[r]<1>{\psset{unit=0.75cm}%
		\begin{pspicture}(-5.2,0)(5.5,5.7)
			\diagnode tl(0,5)[\vcenter{\hbox{\tangleRIh}}]
			\diagnode tr(5,5)[0]
			\diagnode bl(0,1)[\vcenter{\hbox{\tangleRIv}}]
			\diagnode br(5,1)[\vcenter{\hbox{\tangleRIh}}]
			\diagarrow|a|{<->}[tl`tr;0]
			\diagarrow|a|{<->}[tr`br;0]
			\diagarrow[offset=-2pt]|b{labelsep=3pt}|{->}[bl`br;d\,=\,\epsilon\fntCobRId]
			\diagarrow[offset= 2pt]|a{labelsep=3pt}|{<-}[bl`br;h\,=\,-\epsilon^{-1}\fntCobRIh]
			\diagarrow[offset=-2pt]|b{labelsep=3pt}|{->}[tl`bl;\frac\permSS\alpha\left(\!\permMM\fntCobRIfa -\,\permMS\fntCobRIfb\!\right)\,=\,f]
			\diagarrow[offset= 2pt]|a{npos=0.3,labelsep=3pt}|{<-}[tl`bl;g\,=\,\alpha\permMM\permSM\fntCobRIg]
		\end{pspicture}}
The~bracket $\KhBracket*\fntRIx$ is the~mapping cone of the~chain map $\KhBracket*\fntRIv\to\KhBracket*\fntRIh$ induced by edges in the~cube $\KhCube*\fntRIx$ associated to the~distinguished crossing. The chain homotopy equivalences between complexes\linebreak $\KhBracket*\fntRIx$ and $\KhBracket*\fntRIh$ are induced by morphisms of cubes $f\colon\KhCube\fntRIh\twoways\KhCube\fntRIv\cocolon g$ as shown in the~diagram to the~right. Here, $\epsilon$ comes from the~sign assignment used to build $\KhBracket*\fntRIx$, and $\alpha\in\scalars$ is chosen for each component of $f$ and $g$ separately, to make them commute with other edge morphisms in the~cubes. It follows directly from Lemma~\ref{lem:one-comp-is-OK} that $g$ induces a~homogeneous chain map, and for $f$ we have to check that the~two cobordisms have the~same degree. Indeed,
\begin{align}
	\sdeg\Big(\permMS\fntCobRIfb\Big) &= \sdegNum{\phantom-0}{-1} + \sdegNum{\phantom-0}{-2} = \sdegNum{\phantom-0}{-3},\\
	\sdeg\Big(\permMM\fntCobRIfa\Big) &= \sdegNum{1}{0} + \sdegNum{0}{0} + \sdegNum{-1}{-3} + \sdegNum{0}{0} = \sdegNum{\phantom-0}{-3}
\end{align}
after forgetting the~circles not shown in the~diagrams, and placing the~circle drawn in full as the~first one.
	
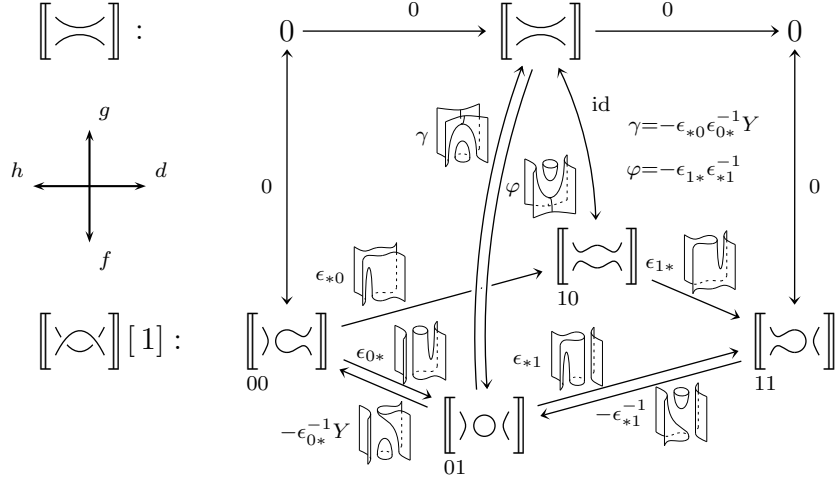
\begin{figure}[t]
	\psset{unit=0.75cm}
	\begin{pspicture}(0,0)(12,8.5)
		\newpsobject{diagarc}{ncarc}{%
				linewidth=0.5pt,doublesep=1.5pt,
				arrowsize=4pt 1,arrowlength=0.8,arrowinset=0.6,
				labelsep=2pt,nodesep=3pt}
		\rput[l](-1,8.0){$\KhBracket*{\vcenter{\hbox{\tangleRIIhhs}}}:$}
		\rput[l](-1,2.5){$\KhBracket*{\vcenter{\hbox{\tangleRIIxx}}}\![\,1]:$}
		\psline{<->}( 0,6.25)(0,4.25)\uput[ur]( 0,6.25){$\scriptstyle g$}\uput[dr](0,4.25){$\scriptstyle f$}
		\psline{<->}(-1,5.25)(1,5.25)\uput[ul](-1,5.25){$\scriptstyle h$}\uput[ur](1,5.25){$\scriptstyle d$}
		\rput( 3.5,8.0){\rnode{t-1}{$0$}}
		\rput( 8.0,8.0){\rnode{t 0}{$\KhBracket*{\vcenter{\hbox{\tangleRIIhhs}}}$}}
		\rput(12.5,8.0){\rnode{t 1}{$0$}}
		\rput( 3.5,2.5){%
			\rlap{\raisebox{-3ex}[0pt][0pt]{$\scriptstyle\mskip\thickmuskip 00$}}%
			\rnode{b00}{$\KhBracket*{\vcenter{\hbox{\tangleRIIvh}}}$}%
		}
		\rput( 7.0,1.0){%
			\rlap{\raisebox{-3ex}[0pt][0pt]{$\scriptstyle\mskip\thickmuskip 01$}}%
			\rnode{b01}{$\KhBracket*{\vcenter{\hbox{\tangleRIIvv}}}$}%
		}
		\rput( 9.0,4.0){%
			\rlap{\raisebox{-3ex}[0pt][0pt]{$\scriptstyle\mskip\thickmuskip 10$}}%
			\rnode{b10}{$\KhBracket*{\vcenter{\hbox{\tangleRIIhh}}}$}%
		}
		\rput(12.5,2.5){%
			\rlap{\raisebox{-3ex}[0pt][0pt]{$\scriptstyle\mskip\thickmuskip 11$}}%
			\rnode{b11}{$\KhBracket*{\vcenter{\hbox{\tangleRIIhv}}}$}%
		}
		\diagline{->}{t-1}{t 0}\naput{$\scriptstyle 0$}
		\diagline{->}{t 0}{t 1}\naput{$\scriptstyle 0$}
		\diagline{->}{b00}{b10}\naput[labelsep=-1pt,npos=0.3]{$\scriptstyle \epsilon_{{*}0}\fntCob{R2-d*0}$}
		\diagline{->}{b10}{b11}\naput[labelsep=-7pt,npos=0.4]{$\scriptstyle \epsilon_{1{*}}\fntCob{R2-d1*}$}
		\diagline[offset= 2pt]{->}{b00}{b01}
				\naput[labelsep=-6pt,npos=0.5]{$\scriptstyle \epsilon_{0{*}}\fntCob{R2-d0*}$}
		\diagline[offset= 2pt]{->}{b01}{b11}
				\naput[labelsep=-1pt,npos=0.3]{$\scriptstyle \epsilon_{{*}1}\fntCob{R2-d*1}$}
		\diagline[offset=-2pt]{<-}{b00}{b01}
				\nbput[labelsep= 0pt,npos=0.5]{$\scriptstyle-\epsilon_{0{*}}^{-1}\permSS\fntCob{R2-h0*}$}
		\diagline[offset=-2pt]{<-}{b01}{b11}
				\nbput[labelsep=-6pt,npos=0.5]{$\scriptstyle-\epsilon_{{*}1}^{-1}\fntCob{R2-h*1}$}
		\diagline{<->}{b00}{t-1}\naput{$\scriptstyle 0$}
		\diagline{<->}{b11}{t 1}\nbput{$\scriptstyle 0$}
		\diagarc[arcangle= 12]{<->}{t 0}{b10}
				\naput[npos=0.3]{$\scriptstyle\id$}
		\diagarc[arcangle= 12,offset= 2pt,border=3\pslinewidth]{->}{b01}{t 0}
				\naput[labelsep=-2pt,npos=0.75]{$\scriptstyle\gamma\fntCob{R2-G}$}
		\diagarc[arcangle= 12,offset=-2pt,border=3\pslinewidth]{<-}{b01}{t 0}
				\nbput[labelsep=-1pt,npos=0.65]{$\scriptstyle\varphi\!\fntCob{R2-F}$}
		\rput[c](10.7,6){$\setlength\arraycolsep{0pt}\begin{array}{rl}
				\scriptstyle\gamma  &\scriptstyle= -\epsilon_{{*}0}^{\vphantom{-1}}\epsilon_{0{*}}^{-1}\permSS\\
				\scriptstyle\varphi &\scriptstyle= -\epsilon_{1{*}}^{\vphantom{-1}}\epsilon_{{*}1}^{-1}
			\end{array}$}
	\end{pspicture}
	\caption{Chain homotopy equivalences for the~second Reidemeister move.}\label{fig:khov-inv-R2}
\end{figure}

\proofpart{Reidemeister II}
Homotopy equivalences for the~second move are shown in Fig.~\ref{fig:khov-inv-R2}. Again, we look on $\KhBracket\fntRIIxx$ as the~total complex of $\KhBracket*\fntRIIvh\to\KhBracket*\fntRIIvv\oplus\KhBracket*\fntRIIhh\to\KhBracket*\fntRIIhv$. Although it looks more challenging, the~way the~morphisms $f$ and $g$ are defined makes the~proof very easy. Indeed, the~morphisms between $\fntRIIhhs$ and $\fntRIIvv$ are compositions of edge morphisms and homotopies: $f_{01} = h_{{*}1}d_{1{*}}$, and $g_{01} = d_{{*}0}h_{0{*}}$. Taking into account the~degree shifts, the~differentials and homotopies preserve $\sdeg$, which implies both $f_{01}$ and $g_{01}$ have the~same degree as the~identity morphisms between $\fntRIIhhs$ and $\fntRIIhh$.

\proofpart{Reidemeister III}
Invariance under the~last move followed from a~strictly algebraic argument: the~complex $\KhBracket*\fntRIIIax$ is the~mapping cone of the~chain map $\KhBracket*\fntRIIIac\colon\KhBracket*\fntRIIIah\to\KhBracket*\fntRIIIav$, and composing it with the~chain homotopy equivalence $f\colon\KhBracket*\fntIIIVert\to\KhBracket*\fntRIIIah$ does not change the~homotopy type of the~mapping cone, see \cite{ChCob}. Hence, $\KhBracket*\fntRIIIax\simeq\cone(\KhBracket*\fntIIIVert\to\KhBracket*\fntRIIIav)$ via degree-preserving chain homotopy equivalences, and similarly for $\KhBracket*\fntRIIIbx$.

\medskip\noindent
This ends the~proof of Theorem~\ref{thm:inv-graded-complex}.\hfill$\square$%

\section{Applications}
\subsection{Reduction of parameters}
Let $\scalars_0\subset\scalars$ be the~subring of degree zero elements. It is generated by $\permMM\permSS$, and as such it isomorphic to $\Zpi := \ZpiLong$. On the~other hand, there is a~ring epimorphism $\scalars\longrightarrow\Zpi$ sending both $\permMM$ and $\permMS$ to $1$, and $\permSS$ to $\pi$. In particular, we can construct $\Zpi\ChCob$, see Remark~\ref{rmk:other-coeffs-for-chcob}.

\begin{lemma}
	The~pair of\/ $\scalars_0$-linear functors $I\colon\Zpi\ChCob\twoways\kChCob_0\cocolon P$,
	\begin{align*}
		I(\Sigma) &:= \Sigma\{0,0\},&
					P(\Sigma\{a,b\}) &:= \Sigma,\\
		I(\pi^k W) &:= \permMM^{a+k}\permSS^k\permMS^b\,W,
		\quad\sdeg W = {\textstyle\sdegNum ab},&
					P(\permMM^p\permSS^q\permMS^r\,W) &:= \pi^q W,
	\end{align*}
	is an~equivalence of categories.
\end{lemma}
\begin{proof}
	Clearly $PI = \id$, and morphisms $\Sigma\{a,b\}\to^{\cdot\permMM^a\permMS^b}\Sigma$ form an~isomorphism $\id\cong IP$.
\end{proof}

Let $\KhCom_\pi(D)$ stand for the~generalized Khovanov complex built in $\catAdd{\Zpi\ChCob}$. Clearly, $\FA\KhCom(D;\Zpi) \cong \Fpi\KhCom_\pi(D)$, where $\Fpi\colon\Zpi\ChCob\to\Mod\Zpi$ is defined similarly to $\FA$.\footnote{
	Think of $\Fpi$ as a~tensor product $\FA\otimes\Zpi$ over $\scalars$.
}

\begin{corollary}\label{cor:KhCom-XYZ-vs-pi}
	The~generalized Khovanov complexes $\KhCom(D)$ and $\KhCom_\pi(D)$ are equivalent link invariants: $\KhCom(D)\simeq\KhCom(D')$ for link diagrams $D$ and $D'$ if and only if $\KhCom_\pi(D)\simeq\KhCom_\pi(D')$.
\end{corollary}

\noindent
There is a~similar equivalence between $\Mod{\scalars,0}$ and $\Mod\Zpi$. Extracting the~degree 0 component $M_0$ of a~$\scalars$-module $M$ results in a~functor $r\colon\Mod{\scalars,0}\longrightarrow\Mod\Zpi$. Dually, given a~$\Zpi$-module $N$ one creates a~$\scalars$-module $i(N) := \raisebox{0.1ex}[1ex][0ex]{$\underset{\mathclap{\scriptscriptstyle(a,b)\in\Z_2{\times}\Z}}\bigoplus$}\ N$, where $(\permMM\permSS)\cdot n = \pi n$, while $\permMM$ and $\permMS$ permute the~copies of $N$ in $i(N)$.

\begin{lemma}\label{lem:Mod-equiv}
	The~pair of functors $i\colon\Mod\Zpi\twoways\Mod{\scalars,0}\cocolon r$ is an~equivalence of categories.
\end{lemma}	
\begin{proof}
	Straightforward.
\end{proof}

\noindent
The~two equivalences intertwine $\FA\colon\kChCob_0\to\Mod{\scalars,0}$ and $\Fpi\colon\Zpi\ChCob\to\Mod\Zpi$, resulting in a~direct connection between $\Kh(D)$ and $\Kh(D;\Zpi)$.

\begin{theorem}[The~reduction of parameters]\label{thm:reduction}
	The~generalized Khovanov complex $\FA\KhCom(D)$, regarded as a~complex of\/ $\Zpi$-modules, decomposes into a~direct sum of subcomplexes
	\begin{equation}\label{eq:Kh-split}
		\FA\KhCom(D)\cong\bigoplus_{\mathclap{\quad(a,b)\in\Z_2\times\Z}\quad}\FA\KhCom(D)_{a,b},
	\end{equation}
	each isomorphic to $\FA\KhCom(D;\Zpi)\cong\Fpi\KhCom_\pi(D)$. The~action of\/ $\scalars$ is given by isomorphisms
	\begin{equation}
		\begin{cases}
			\permMM\colon\FA\KhCom(D)_{a,b}\to<2em>^{\id}\FA\KhCom(D)_{a+1,b},\\
			\permSS\colon\FA\KhCom(D)_{a,b}\to<2em>^{\cdot\pi}\FA\KhCom(D)_{a+1,b},\\
			\permMS\colon\FA\KhCom(D)_{a,b}\to<2em>^{\id}\FA\KhCom(D)_{a,b+1}.
		\end{cases}
	\end{equation}
\end{theorem}
\begin{proof}
	The~decomposition follows from Theorem~\ref{thm:inv-graded-complex}, so it remains to compute the~degree zero subcomplex. First, $r(M)$ is naturally isomorphic to $M\otimes\Zpi$ via $m\mapsto m\otimes 1$. Indeed, this map is linear over $\Zpi$, and its inverse sends $m\otimes 1$, with $\sdeg(m) = \sdegNum ab$, into $\permMM^a\permMS^b m$. Hence, $\FA\KhCom(D)_{0,0}$ is naturally isomorphic to $\FA\KhCom(D)\otimes\Zpi = \FA\KhCom(D;\Zpi)$.
\end{proof}

Given a~graded ring automorphism $\varphi\in\Aut_0(\scalars)$ we can replace the~chronological parameters $\permMM$, $\permSS$, and $\permMS$ with its images under $\varphi$, resulting in a~graded category $\scalars_\varphi\ChCob_0$ and a~chronological TQFT $\FA_{\!\varphi}\colon\scalars_\varphi\ChCob_0 \to \Mod{\scalars,0}$. As before, given a~link diagram $D$ we can construct the~generalized Khovanov complex $\KhCom_\varphi(D)$ in $\catAdd{\scalars_\varphi\ChCob}$. In the~view of Corollary~\ref{cor:KhCom-XYZ-vs-pi}, the~complexes $\FA\KhCom(D)$ and $\FA_{\!\varphi}\KhCom_\varphi(D)$ are equivalent link invariants if $\varphi(\permMM\permSS)=\permMM\permSS$ (i.e.\ if $\varphi|_{\scalars_0}=\id$). We shall now show they are in fact isomorphic.


\begin{proposition}\label{prop:change-of-params}
	Assume $\varphi(\permMM\permSS) = \permMM\permSS$. Then the~complexes of\/ $\scalars$-modules $\FA\KhCom(D)$ and $\FA_{\!\varphi}\KhCom_\varphi(D)$ are isomorphic for any link diagram $D$.
\end{proposition}
\begin{proof}
	Decompose the~complexes as in Theorem~\ref{thm:reduction}. Then $\FA\KhCom(D)_{0,0}$ and $\FA_{\!\varphi}\KhCom_\varphi(D)_{0,0}$ are complexes of free $\Zpi$-modules, and $\varphi$ induces an~isomorphism between them. Indeed, $\pi$ acts on both complexes as multiplication by $\permMM\permSS=\varphi(\permMM\permSS)$. Thence, it is enough to extend the~equality in a~$\scalars$-linear way. Explicitly,
	\begin{equation}
		\FA\KhCom(D) \ni u \mapsto
			\big(\tfrac{\varphi(\permMM)}\permMM\big)^a
			\big(\tfrac{\varphi(\permMS)}\permMS\big)^b u \in \FA_\varphi\KhCom_\varphi(D)
	\end{equation}
	for a~generator $u=\vvv{i_1,*\dots,i_k}$ in degree $\sdeg(u)=\sdegNum ab$.\footnote{
		Here, $\sdeg(u)$ is the~degree of $u$ as an~element of graded $\FA\KhCom$, and it can be different as when $u$ is regarded as an~element $A^{\otimes k}$.
	}
\end{proof}

Denote by $\scalars_\varphi$ the~ring $\scalars$ with a~module structure twisted by $\varphi$, i.e.\ $k\cdot x := \varphi(k)x$. Every $\scalars$-module structure on $\Z$ can be obtained by taking a~tensor product $\scalars_\varphi\otimes\Zev$ or $\scalars_\varphi\otimes\Zodd$ for an~automorphism $\varphi$ fixing $\permMM\permSS$. For instance, if $\varphi(\permMM)=-\permMM$ and likewise for $\permSS$ and $\permMS$, then each parameter acts on $\Z':=\scalars_\varphi\otimes\Zev$ as $-1$.

\begin{corollary}
	Given a~$\scalars$-module structure on $\Z$, the~homology $\Kh(L;\Z)$ is either the~even Khovanov homology, if $\permMM\permSS$ acts on $\Z$ as identity, or the~odd Khovanov homology otherwise.
\end{corollary}

\begin{remark}
	The~even and odd Khovanov homology are not equivalent. Hence, the~condition on $\varphi$ in Proposition~\ref{prop:change-of-params} is necessary.
\end{remark}
	
\begin{remark}
	Theorem~\ref{thm:reduction} is true for any chronological Frobenius system $(R,A)$ with $R$ supported in a~single weight 0. In particular, we can take the~algebra of dotted cobordisms $(R_\bullet, A_\bullet)$ \cite{ChCob}, as $R_\bullet$ is generated over $\scalars$ by $h$ and $t$ of degrees $\deg h = (-1,-1)$ and $\deg t = (-2,-2)$ respectively.
\end{remark}

\subsection{Duality}
Choose a~link diagram $D$. The~Khovanov homology of the~\emph{mirror image} $D^!$ is dual to the~one of $D$ \cite{KhHom}. Namely, there is an~isomorphism of complexes
\begin{equation}
	\Fev\KhCom(D^!) \cong \Hom(\Fev\KhCom(D),\Z),
\end{equation}
which follows from the~following two observations.
\begin{enumerate}
	\item\wrapfigure[r]{\psset{unit=3mm,labelsep=2pt}
	\begin{pspicture}(-6,-3)(6,4)
	\rput[c](-5,0){\Rnode{hor}{\begin{pspicture}[shift=-0.5](-1,-1)(1,1)
			\psbezier(-1,-1)(0,-0.1)(0,-0.1)(1,-1)
			\psbezier(-1, 1)(0, 0.1)(0, 0.1)(1, 1)
	\end{pspicture}}}
	\rput[c](5,0){\Rnode{vert}{\begin{pspicture}[shift=-0.5](-1,-1)(1,1)
			\psbezier(-1,-1)(-0.1,0)(-0.1,0)(-1,1)
			\psbezier( 1,-1)( 0.1,0)( 0.1,0)( 1,1)
	\end{pspicture}}}
	\rput[c](0,3){\Rnode{NWSE}{\begin{pspicture}[shift=-1](-1,-1)(1,1)
			\psline(-1,-1)(1,1)
			\psline[border=5\pslinewidth](-1,1)(1,-1)
	\end{pspicture}}}
	\rput[c](0,-3){\Rnode{NESW}{\begin{pspicture}(-1,-1)(1,1)
			\psline(-1,1)(1,-1)
			\psline[border=5\pslinewidth](-1,-1)(1,1)
	\end{pspicture}}}
	\diagline{->}{NWSE}{hor} \nbput{$\scriptstyle0$}
	\diagline{->}{NWSE}{vert}\naput{$\scriptstyle1$}
	\diagline{->}{NESW}{hor} \naput{$\scriptstyle1$}
	\diagline{->}{NESW}{vert}\nbput{$\scriptstyle0$}
	\diagline[linestyle=dotted]{<->}{NWSE}{NESW}\naput{$\scriptstyle!$}
	\end{pspicture}}
	Resolutions of crossings in $D^!$ are those of $D$, but with type $0$ and type $1$ interchanged, as illustrated to the~right. In particular, the~cube $\KhCube{D^!}$ looks like $\KhCube{D}$, but with all arrows reversed.
	\item The~Khovanov's algebra $A$ is self-dual: $A^*\cong A$ via $v_\pm^*\mapsto v_\mp$.
\end{enumerate}
A~similar phenomenon occurs in the~generalized case \cite{ChCob} with a~few differences. The~algebra $A$ over $\scalars$ is not strictly self-dual: $A^*\cong \overline A$ is algebras over $\overline\scalars$, where we exchange the~roles of $\permMM$ and $\permSS$.\footnote{
	In other words, $\overline\scalars = \scalars_\varphi$, where $\varphi(\permMM) = \permSS$, $\varphi(\permSS) = \permMM$, and $\varphi(\permMS)=\permMS$.
} For instance,
\begin{gather}
	\Delta^*(v_+^*\otimes v_-^*) = \permSS\permMS v_+^*, \quad\text{and}	\\
	\mu^*(v_-^*) = v_+^*\otimes v_-^* + \permMM\permMS v_-^*\otimes v_+^*.
\end{gather}
Likewise, the~duality between cubes $\KhCube{D^!}$ and $\KhCube{D}$ is realized by an~operation on chronological cobordisms $(\blank)^*\colon\kChCob \to \overline\scalars\ChCob$ that `reverses' the~chronology, i.e.\ $(W,\tau)^* := (W,\tau^*)$ with $\tau^*(t) := 1-\tau(t)$. Reversing a~cobordism permutes its degree components, $\chdeg W^* = (b,a)$ if $\chdeg W = (a,b)$, but it also intertwines the two disjoint unions, $(W\rdsum W')^* = W^*\ldsum W'^*$. This explains why the~roles of $\permMM$ and $\permSS$ are exchanged, but the~role of $\permMS$ is preserved.

When reversing the~chronology of a~cobordism one must also take care of orientations of critical points. Indeed, an~orientation of a~critical point $p\in W$ induces an~orientation of the~stable part of $T_pW^*$. We choose for the~unstable part the~complementary orientation with respect to the~outward orientation of the~cobordism $W$. Diagrammatically, color each region in the~complement of $W$ black or white, so that the~unbounded region is white and regions with same colors do not meet; then for saddle point $p$ rotate the~framing arrow in $W^*$ clockwise if the~region below $p\in W$ is white, and anticlockwise otherwise:
\begin{gather*}
	\psset{unit=8mm}
	\left(\begin{centerpict}(-0.1,-0.1)(1.3,1.3)
		\rput[tl](0,1.2){\textnormal{\scriptsize b}}%
		\rput(0.6,0.7){\textnormal{\scriptsize w}}%
		\cobordism[2](0,0)(M-R)
	\end{centerpict}\right)^{\!\!*} = \textcobordism[1](S-B)
		\hskip 2cm
	\left(\begin{centerpict}(-0.1,-0.1)(1.3,1.3)
		\rput[bl](0,0.0){\textnormal{\scriptsize b}}%
		\rput(0.6,0.5){\textnormal{\scriptsize w}}%
		\cobordism[1](0.4,0)(S-B)
	\end{centerpict}\right)^{\!\!*} = \textcobordism[2](M-R)
		\\[1ex]
	\psset{unit=8mm}
	\left(\begin{centerpict}(-0.1,-0.1)(1.3,1.3)
		\rput[tl](0,1.2){\textnormal{\scriptsize w}}%
		\rput(0.6,0.7){\textnormal{\scriptsize b}}%
		\cobordism[2](0,0)(M-R)
	\end{centerpict}\right)^{\!\!*} = \textcobordism[1](S-F)
		\hskip 2cm
	\left(\begin{centerpict}(-0.1,-0.1)(1.3,1.3)
		\rput[bl](0,0.0){\textnormal{\scriptsize w}}%
		\rput(0.6,0.5){\textnormal{\scriptsize b}}%
		\cobordism[1](0.4,0)(S-B)
	\end{centerpict}\right)^{\!\!*} = \textcobordism[2](M-L)
\end{gather*}
Since we want the~duality functor to be coherent with \eqref{rel:reverse-orientation} there is no choice left for births and deaths:
\begin{gather*}
	\psset{unit=1cm}
		\left(\begin{centerpict}(-0.1,-0.1)(0.5,0.7)
			\rput[br](0.05,0){\textnormal{\scriptsize b}}%
			\rput(0.2,0.4){\textnormal{\scriptsize w}}%
			\cobordism[0](0.4,0)(sB)
		\end{centerpict}\right)^{\!\!*} = \textcobordism[1](sD+)
			\hskip 2cm
		\left(\begin{centerpict}(-0.1,-0.1)(0.5,0.7)
			\rput[tr](0.05,0.6){\textnormal{\scriptsize b}}%
			\rput(0.2,0.2){\textnormal{\scriptsize w}}%
			\cobordism[1](0,0)(sD+)
		\end{centerpict}\right)^{\!\!*} = \textcobordism[0](sB)
			\hskip 2cm
		\left(\begin{centerpict}(-0.1,-0.1)(0.5,0.7)
			\rput[tr](0.05,0.6){\textnormal{\scriptsize b}}%
			\rput(0.2,0.2){\textnormal{\scriptsize w}}%
			\cobordism[1](0,0)(sD-)
		\end{centerpict}\right)^{\!\!*} = \permSS\textcobordism[0](sB)
			\\[1ex]
	\psset{unit=1cm}
		\left(\begin{centerpict}(-0.1,-0.1)(0.5,0.7)
			\rput[br](0.05,0){\textnormal{\scriptsize w}}%
			\rput(0.2,0.4){\textnormal{\scriptsize b}}%
			\cobordism[0](0.4,0)(sB)
		\end{centerpict}\right)^{\!\!*} = \textcobordism[1](sD-)
			\hskip 2cm
		\left(\begin{centerpict}(-0.1,-0.1)(0.5,0.7)
			\rput[tr](0.05,0.5){\textnormal{\scriptsize w}}%
			\rput(0.2,0.2){\textnormal{\scriptsize b}}%
			\cobordism[1](0,0)(sD-)
		\end{centerpict}\right)^{\!\!*} = \textcobordism[0](sB)
			\hskip 2cm
		\left(\begin{centerpict}(-0.1,-0.1)(0.5,0.7)
			\rput[tr](0.05,0.5){\textnormal{\scriptsize w}}%
			\rput(0.2,0.2){\textnormal{\scriptsize b}}%
			\cobordism[1](0,0)(sD+)
		\end{centerpict}\right)^{\!\!*} = \permSS\textcobordism[0](sB)
\end{gather*}

We showed in \cite{ChCob} that $\KhCubeSigned{D^!}\epsilon$, regarded as an~object in $\Kom(\overline\scalars\ChCob)$, is the~image of $\KhCubeSigned D\epsilon$ under the~above operation, which implies $\F_{\!\overline\scalars}\KhCom_{\overline\scalars}(D^!)\cong\F\KhCom(D)^*$. The~results of the~previous section allows us to switch $\overline\scalars$ back to $\scalars$

\begin{theorem}[Duality for generalized Khovanov homology]\label{thm:duality}
	Given a~link diagram $D$ and its mirror image $D^!$ there is an~isomorphism of complexes
	\begin{equation}\label{eq:duality-isomorphism}
		\FA\KhCom(D^!) \cong \FA\KhCom(D)^*,
	\end{equation}
	where $(C^*)^i := \Hom(C^{-i},\scalars)$ for a~chain complex $C$. In particular, the~odd Khovanov homology $\OKh(L)$ of a~link $L$ is dual to $\OKh(L^!)$, and similarly for $\Kh_\pi(L)$ and $\Kh_\pi(L^!)$.
\end{theorem}
\begin{proof}
	Proposition~\ref{prop:change-of-params} and the~discussion above give a~sequence of isomorphisms
	\begin{equation}
		\FA\KhCom(D^!) \cong \FA_{\overline\scalars}\KhCom_{\overline\scalars}(D^!) \cong \FA\KhCom(D)^*.
	\end{equation}
	The~cases of $\OKh$ and $\Kh_\pi$ follows from an~isomorphism $\Hom(F,\scalars)\otimes R \cong\Hom(F\otimes R, R)$ for any a~free module $F$ and a~ring homomorphism $\scalars\to<1em> R$.
\end{proof}

The~duality isomorphism \eqref{eq:duality-isomorphism} is given explicitly as
\begin{equation}
	\FA(D^!_\zeta)\ni u \longmapsto (\permMM\permSS)^a\,u^* \in \FA(D_{\bar\zeta})^*,
\end{equation}
where $u = \vvv{i_1,*\dots,i_k}$ has degree $\sdeg u = \sdegNum ab$. For the~other version of Khovanov homology, simply replace $\permMM\permSS$ with either $\pi$, for the~unified homology, or $(-1)$ for the~odd one. Note the~role of the~splitting degree: although it does not descend directly to $\OKh(L)$ nor $\Kh_\pi(L)$, it controls the~duality isomorphism.

\end{document}